\documentclass{amsart}
\usepackage{amssymb,stmaryrd}
\usepackage{amsfonts}
\usepackage{amstext}
\usepackage{algorithmic}
\usepackage{algorithm}
\usepackage{graphicx}
\usepackage{epstopdf}
\usepackage[all]{xy}
\usepackage{MnSymbol}
\usepackage{slashed}

\usepackage{amsaddr}

\numberwithin{equation}{section}
\newtheorem{theorem}{Theorem}[section]

\newtheorem{proposition}[theorem]{Proposition}

\theoremstyle{definition}
\newtheorem{definition}[theorem]{Definition}

\theoremstyle{remark}

\newcommand{\Hom}{{\rm{Hom}}}

\renewcommand{\S}{{\mathcal{S}}}

\newcommand{\C}{{\mathbb{C}}}

\newcommand{\J}{{\mathcal{J}}}

\newcommand{\pdol}{\overline{\partial}}
\newcommand{\pd}{\partial}

\newcommand{\<}{{\langle}}

\renewcommand{\)}{{)}}
\renewcommand{\>}{{\rangle}}

\newcommand{\CS}{{\mathcal{S}}}

\newcommand{\CJ}{{\mathcal{J}}}
\newcommand{\CH}{{\mathcal{H}}}
\newcommand{\wedgeq}{{\wedge\kern-5pt\cdot}}

\newcommand{\tens}{\otimes}

\newcommand{\id}{{\rm id}}

\newcommand{\extd}{{\rm d}}
\newcommand{\del}{{\partial}}
\newcommand{\eps}{\epsilon}
\newcommand{\ev}{{\rm ev}}
\newcommand{\coev}{{\rm coev}}

\newcommand{\la}{{\triangleright}}

\newcommand{\dirac}{{  \slashed{D} }}
\newcommand{\Y}{{w}}

\renewcommand{\imath}{\mathrm{i}}

\begin{document}

\title{Spectral triples from bimodule connections and Chern connections} 


\author{Edwin Beggs \& Shahn Majid}
\address{Dept of Mathematics, Swansea University\\ Singleton Parc, Swansea SA2 8PP\\
{\ }+\\ Queen Mary, University of London\\
School of Mathematics, Mile End Rd, London E1 4NS, UK}

\email{s.majid@qmul.ac.uk, e.j.beggs@swansea.ac.uk (corresponding author)}

\begin{abstract} We give a geometrical construction of Connes spectral triples or noncommutative Dirac operators $\dirac$ starting with a bimodule connection on the proposed spinor bundle. The theory is applied to the example of $M_2(\C)$, and also applies to the standard $q$-sphere and the $q$-disk with the right classical limit and all properties holding except for  $\CJ$ now being a twisted isometry.  We also describe a noncommutative 
Chern construction from holomorphic bundles which in the $q$-sphere case provides the relevant bimodule connection. \end{abstract}

\maketitle 

\section{Introduction}

A main difference between the well-known Connes approach to noncommutative geometry coming out of cyclic cohomology and the more constructive `quantum group' approach to noncommutative geometry lies in the attitude towards the Dirac operator. In Connes' approach this is defined axiomatically as an operator $\dirac$ on a Hilbert space which plays the role of Dirac operator on a spinor bundle and which is the starting point for Riemannian geometry, while in the quantum groups approach one builds up the geometry layer by layer starting with the differential algebra structure and often (but not necessarily) guided by quantum group symmetry, and arrives at $\dirac$ as an endpoint, normally after the Riemannian structure.  This approach also should contain $q$-deformed and quantum group-related examples but it is known that these may take us beyond Connes axioms if we want to have the correct classical limit. For example, for the standard $q$-sphere where the construction in \cite{DS-sphere} meets Connes axioms at some algebraic level but has spectral dimension 0 (the eigenvalues of the Dirac operator distributes in a typical manner for a zero dimensional manifold) and hence do not have the correct classical limit.

The present paper joins up these two approaches, namely we show how within the constructive approach we can naturally obtain spectral triples, at least up to issues of functional analysis, from a bimodule connection on a chosen vector bundle (thought of as a `spinor bundle'), having fixed a first order differential calculus for our space and a `Clifford action' $\la$ of its 1-forms on the bundle. The latter plays the role of the Clifford structure. Our construction is still quite general and we don't assume that the bundle is associated to a quantum frame bundle and connection induced by a quantum `spin' connection on it  as per the classical case, although that will be the case in the $q$-sphere example. 

An outline of the paper is as follows.  In Section~2.1, we recall Connes' axioms \cite{Con,ConMar} for a real spectral triple. Then in Section~2.2 we provide our main result, Theorem~\ref{sptripres}, which constructs examples of these from bimodule connections at an algebraic level, i.e.\ before worrying about adjoints. Section~2.3 establishes further constraints on the bimodule connection and inner product data to have $\dirac$ hermitian and $\J$ an (antilinear) isometry. Section~2.4 completes the general theory with an explanation of how varying the bimodule connection amounts to an inner fluctuation of the spectral triple in the sense of Connes\cite{ConMar}. 

One of the first ingredients in Section~2.2 is that the `commutativity condition' in Connes' axioms (see (4) in our recap below) can be seen as making the Hilbert space $\CS$ a bimodule \cite{LordDirac}, see also \cite{Barrmatrix}. However, our notion of bimodule connection means a single (say, left) connection $\nabla$ which admits a modified right-connection rule via a generalised braiding \cite{Mou,DV1,DV2,MMM,Sitarz,BegMa3,BegMa4,BegMa5,MaTao}. This allows for connections on tensor products of bimodules which will be critical for what follows and is very different from what is meant by `bimodule connection' in \cite{LordDirac}, which comes from \cite{CuntzQuillen} and uses two unrelated connections, one left and one right, on a bimodule. Classically, the latter reduces to defining two unrelated connections on the same bundle and is not what we need. Specifically, the lack of relation between the left and right structures means that the antilinear $\J$ operator for the reality condition for Connes' definition of Dirac operator could not be studied. In the context of what we mean by bimodule connections, another  main tool in Section~2.2 is a conjugate bimodule whereby
the antilinear map  $\J:\CS\to \CS$ is formulated in terms of a linear map 
 $j:\CS\to \overline{\CS}$. We use our previous work \cite{BegMa3} for the conjugate bimodule connection and related matters.  Although one could view the use of bar categories here as a bookkeeping device to keep explicit track of anti/linearity, it is essential for tensor product operations like $\id\tens j$ to make sense.  In the context of general monoidal categories, the idea of bar category can be less trivial \cite{BegMa2},  but it  is very useful even in the present case of complex vector spaces and someantilinear maps. 

Section~3 shows how the theory works on three examples. Section~3.1 covers the finite geometry of $2\times 2$ matrices $M_2(\C)$ as `coordinate algebra'. This is of course very well studied and we refer to \cite{Barrmatrix} for a recent treatment of spectral triples here. In our approach we start with a natural $*$-differential calculus $\Omega^1$ which is 2-dimensional over the algebra. As it happens we take the same bimodule for $\CS$, i.e.\ 2-spinors. We take a natural choice of $\la$ in this context and fixing this data we find a unique bimodule connection that meets our requirements of Section~2. This results in a single spectral triple which we compute as $\dirac={1\over 2}\gamma^2\tens[\gamma^2,\ ]-{1\over 2}\gamma^1\tens[\gamma^1,\ ]$ where $\gamma^i=\imath\sigma^i$ in terms of Pauli matrices. The commutators are inner derivations or `vector fields' on $M_2(\C)$ and uniqueness means that fluctuations of this would entail a change of either the differential structure or the Clifford structure.  

Section~3.2 covers the $q$-sphere $\C_q[S^2]$ with the geometrically correct spin bundle $\CS=\CS_+\oplus\CS_-$ given by $q$-monopole sections of charges $\pm 1$ as used in \cite{Ma:rieq}. This uses the standard 2D differential calculus coming from the 3D one\cite{Wor} on $\C_q[SU_2]$, 
 a Clifford action $\la$ given by the holomorphic structure introduced in \cite{Ma:rieq} and a $q$-monopole  principal connection \cite{BrzMaj:gau}, all of which led to a $q$-deformed $\dirac$ in a quantum frame bundle approach. Our new result is that the relevant covariant derivative on $\CS$ is in fact a bimodule connection and we find a $\CJ$ and inner product (given by the Haar integral) so that all the axioms (1)-(6) of a real spectral triple of dimension 2 are satisfied at the pre-functional analysis level except for one: we find that $\CJ$ is necessarily not an isometry but some kind of twisted $q$-isometry in the sense
\[ \<\!\<\J(\phi),\J(\psi)\>\!\>=q^{\pm 1} \<\!\<\varsigma^{-1}(\psi),\phi\> \!\>,\quad\forall \phi,\psi\in \CS_\pm\]
where the brackets are the Hilbert space inner product and $\varsigma$ is the automorphism that makes the Haar integral a twisted trace in the sense of \cite{Murphy}. We identified $\CS_\pm$ with degree $\mp1$ subspaces of $\C_q[SU_2]$ under the $U(1)$ action of the quantum principal bundle. More precisely, we obtain a 1-parameter family of $\dirac$ where a parameter $\beta$ extends  the Clifford action from the canonical choice $\beta=1$  in \cite{Ma:rieq}. Our construction is different from another attempt at the $q$-sphere Dirac operator with 2D spinor space \cite{DLPS-sphere}, where the `first order condition' (see (6) in our recap below) had to be weakened to hold up to compact operators, which is not our case. Section~3.3 is our final example, the quantum disk $\C_q[D]$ as in \cite{klilesdisk}, where we find again that everything works up to completions to give a dimension 2 spectral triple, for our choice of bimodule connection, aside from $\CJ$ being required to be a twisted isometry. The Clifford structure is similar the the $q$-sphere case and we again obtain a moduli space of examples as we vary a real parameter. 

Section~4 returns to the general theory with a noncommutative  framework for  holomorphic bundles  and their associated Chern connections along the lines of \cite{Cherncomplex}. At least in the nice case of the $q$-sphere this provides a more direct geometric route to the bimodule connection that we used for the $\dirac$ operator as well as the Levi-Civita connection in \cite{Ma:rieq} (i.e. without going through the frame bundle theory). This is computed along with the other examples in Section~5. 

We also note \cite{RieffelResist} which has a similar starting point of a differential algebra equipped with a metric and which contains some steps towards a more analytic treatment.

\section{Connes spectral triple}

This section starts with a short recap of Connes' axioms of a spectral triple and then proceeds to our main results about the construction of these from bimodule connections. 

\subsection{Real spectral triples at an algebraic level}

At the pre-functional analysis level, a real spectral triple in dimension $n$ mod 8 consists of certain data \cite{ConMar} which we list as follows:

(1a) A  Hilbert space $\CH$, with inner product $\<\!\<\ ,\ \>\!\>$ antilinear in the 1st argument. A faithful representation of the a $*$-algebra $A$ on $\CH$ such that
for all $a\in A$ and $\phi,\psi\in \CH$,
\[ \<\!\<a^*.\psi,\phi\>\!\>=\<\!\<\psi, a.\phi\>\!\>\]

(1b) Operators $\gamma, \dirac$ (both linear) and $\J$ (antilinear) on $\CH$ obeying $\<\!\<\J\psi, \J\phi\>\!\>=\<\!\<\phi,\psi\>\!\>$, $\gamma^*=\gamma$ and $\dirac^*=\dirac$. (For odd dimension we may take $\gamma$ to be the identity.) 

(2)  $\J^2=\eps$, $\J\gamma=\eps''\gamma \J$,  $\gamma^2=1$, $[\gamma, a]=0$

(3) $\dirac\gamma=(-1)^{n-1}\gamma \dirac$. 

(4) $[a, \J b\J^{-1}]=0$ for all $a,b\in A$

(5)  $\J \dirac=\eps' \dirac \J$

(6) $[[\dirac,a],\J b\J^{-1}]=0$ for all $a,b\in A$

\medskip

The signs $\eps,\eps',\eps''$ in $\{+1,-1\}$ are taken from a table according to $n$ mod 8:

\begin{center}
{\renewcommand{\arraystretch}{1.3}
\begin{tabular}{|c|c|c|c|c|c|c|c|c|}
\hline $n$ & 0 & 1 & 2 & 3 & 4 & 5 & 6 & 7 \\ 
\hline $\epsilon$ & 1 & 1 & $-1$ & $-1$ & $-1$ & $-1$ & 1 & 1 \\ 
\hline $\epsilon'$ & 1 & $-1$ & 1 & 1 & 1 & $-1$ & 1 & 1 \\ 
\hline $\epsilon''$ & 1 &  & $-1$ &  & 1 &  & $-1$ &  \\ 
\hline 
\end{tabular} 
}
\end{center}

We have grouped the axioms here into (1) that relate to the Hilbert space structure and ultimately to functional analysis, and the remainder which are
more algebraic.

\subsection{Construction of spectral triples from connections}
Take a star algebra $A$ with a star differential calculus
$(\Omega,\extd,\wedge)$, and a left $A$-module $\CS$. The first proposition is also the starting point of \cite{LordDirac}. 

\begin{proposition} \label{condd1}
Suppose we are given an antilinear map $\J:\CS\to \CS$ and a linear map $\gamma:\CS\to \CS$ satisfying properties (2) and (4). Then there is a bimodule structure on $\CS$, with right action,  for $\psi\in\CS$ and $a\in A$
\[ \psi.a= \J a^*\J^{-1}.\psi\ ,\]
and $\gamma$ is a bimodule map. 
\end{proposition}
\begin{proof} 
Property (4) states that the left and right actions commute. The condition $[\gamma,a]=0$ shows that $\gamma$ is a left module map. Then
\[
\gamma\J a^*\J^{-1}=\epsilon''\,\J \gamma a^*\J^{-1}=\epsilon''\,\J\,a^*\, \gamma \J^{-1}=\J\,a^*\J^{-1} \gamma \ ,
\]
so $[\gamma,\J a^*\J^{-1}]=0$, so $\gamma$ is a right module map.
\end{proof}

Now we assume the conditions for Proposition~\ref{condd1}, and examine some of the other conditions, given a particular construction for $\dirac$.
However first, we need to define a bimodule connection.

\begin{definition}
A left connection $\nabla_\CS:\CS\to\Omega^1\tens_A \CS$ on $\CS$ is a  linear map obeying the left Leibniz rule
\[
\nabla_\CS(a.\phi)=\extd a\tens\phi+a.\nabla_\CS(\phi)\ ,
\]
for $a\in A$ and $\phi\in \CS$. A left bimodule connection is a pair $(\nabla_\CS,\sigma_\CS)$ where $\nabla_\CS$ is a left connection and 
$\sigma_\CS:\CS\tens_A \Omega^1\to\Omega^1\tens_A \CS$ is a bimodule map obeying
\[
\sigma_\CS(\phi\tens\extd a)=\nabla_\CS(\phi.a)-\nabla_\CS(\phi).a\ .
\]
\end{definition}

Note that we have a single connection, with a left Leibniz rule and a modified right Leibniz rule. This is the definition of bimodule connection used in  \cite{Mou,DV1,DV2,MMM,Sitarz,BegMa3,BegMa4,BegMa5,MaTao} among others, and is defined in that manner so as to enable the tensor product of connections.

\begin{proposition}\label{Diracnabla}
Suppose that $(\nabla_\CS,\sigma_\CS)$ is a left bimodule connection on $\CS$, and that $\la:\Omega^1\tens_A \CS\to 
\CS$ is a left module map. If we define $\dirac=\la\circ\nabla_\CS$ then $[\dirac,a]\phi=\extd a\la \phi$. Then (6) is equivalent to $\la$ being a bimodule map, and
(5) is equivalent to 
\begin{align*}
\epsilon'\, \J[\dirac,a^*]\J^{-1}\phi=\la(\sigma_\CS(\phi\tens \extd a))\ .
\end{align*}
\end{proposition}
\noindent\textbf{Proof:}\quad 
The first statement is given by
\[
\la\circ\nabla_\CS(a.\phi)=\la(\extd a\tens \phi+a.\nabla_\CS \phi)
=\extd a\la \phi+a.(\la)(\nabla_\CS \phi)\ ,
\]
as $\nabla_S$ is a connection, and the comment on bimodule maps is then immediate. If $\nabla_S$ is a bimodule connection we have $\dirac(\phi.a)=\la(\sigma_\CS(\phi\tens \extd a)+(\dirac\phi).a$. On the other hand, using (5) for the 2nd equality,
\begin{align*}
\dirac(\phi.a)=&\ \dirac(\J a^*\J^{-1}\phi)=\epsilon'\, \J \dirac a^*\J^{-1} \phi = \epsilon'\, \J[\dirac,a^*]\J^{-1}\phi+\epsilon'\, \J a^*\dirac\J^{-1}\phi \cr
=&\ \epsilon'\, \J[\dirac,a^*]\J^{-1}\phi+ \J a^*\J^{-1}\dirac \phi
= \epsilon'\, \J[\dirac,a^*]\J^{-1}\phi+ (\dirac \phi).a\ . \qquad\largesquare
\end{align*}
so the condition stated follows from (5). The argument is clearly reversible and (5) holds if the condition stated holds for all $a,\phi$. 

\medskip
Now we examine how to satisfy the conditions on  $\gamma$ in terms of the connection and the `Clifford action' $\la$:

\begin{proposition}
If there is a bimodule map $\gamma:\CS\to \CS$ with $\gamma^2=\id$, 
which intertwines the connection $\nabla_\CS$ (i.e.\ $\nabla_\CS\gamma=(\id\tens\gamma)\nabla_\CS$), and has
\[
\gamma\circ\la=-\la\circ(\id\tens\gamma):\Omega^1\tens_A \CS\to \CS\ ,\
\J\circ\gamma=\epsilon''\,\overline{\gamma}\circ \J:\CS\to\CS\ ,
\]
then $(\dirac,\J,\gamma)$ satisfies all the conditions which include $\gamma$ in (2-6) for an even dimensional spectral triple. 
\end{proposition}
\noindent\textbf{Proof:}\quad There is only one nontrivial thing to check,
\[
\gamma\,\dirac(\phi)=\gamma\circ(\la)\nabla_\CS\phi=-(\la)\circ(\id\tens\gamma)\nabla_\CS\phi=
-(\la)\circ\nabla_\CS\gamma\phi=-\dirac(\phi)\ .\quad\largesquare
\]

\medskip
At first sight it might seem that satisfying the conditions for $\J$ would be very similar to the case for $\gamma$. However, this is not the case. The problem is that $(\id\tens\J)\nabla_\CS$ is not even defined. In a tensor product over the complex numbers, we have
$\mathrm{i}\,\xi\tens \phi=\xi\tens \mathrm{i}\, \phi\in\Omega^1\tens_A \CS$. Now applying $(\id\tens\J)$ to this gives
$\mathrm{i}\,\xi\tens \J\phi=-\xi\tens \mathrm{i}\, \J\phi$, a contradiction unless both sides vanish. All this is before we actually look at the $\tens_A$ part, and find more problems in that elements of $A$ are multiplied on the wrong side. Basically, tensor products and antilinear maps do not mix. 
Our problems are resolved if we are more careful and use the 
 \textit{conjugate} of a bimodule.

\begin{definition}\cite{BegMa3} 
The conjugate of an $A$-bimodule $E$ is written $\overline{E}$, and is identical to $E$ as a set with addition. We denote an element of $\overline{E}$ by $\overline{e}$ where $e\in E$, so that they are not confused. The complex vector space structure, for $e,g\in E$ and $\lambda\in\C$ is
\[
\overline{e}+\overline{g}=\overline{e+g}\ ,\quad \lambda\,\overline{e}=\overline{\lambda^*\,e}\ .
\]
The $A$-bimodule structure is given by a change of side, for $a\in A$,
\[
a.\overline{e}=\overline{e.a^*}\ ,\quad \overline{e}.a=\overline{a^*.e}\ .
\]

For bimodules $E,F$, a bimodule map $\theta:E\to F$ gives another bimodule map $\overline{\theta}:\overline{E}\to \overline{F}$
by $\overline{\theta}(\overline{e})=\overline{\theta(e)}$.
There is also a well defined bimodule map $\Upsilon:\overline{E\tens_A F}\to \overline{F}\tens_A \overline{E}$ flipping the order, defined by, for $e\in E$ and $f\in F$,
given by
\[
\Upsilon(\overline{e\tens f})=\overline{f}\tens_A \overline{e}
\]
\end{definition}

Note that we \textit{do not use bar as a complex conjugation operation}, on elements it is purely a bookkeeping notation for the antilinear identity map. In fact, if we want to take the complex conjugate of $\lambda\in\C$, as above, we write it as $\lambda^*\in\C$ to avoid confusion. The alternative, as stated above, is to be incapable of incorporating antilinear maps into tensor products. With this notation, an antilinear map can be regarded as a linear map into the conjugate. The flip map $\Upsilon$ simply implements a change of order implicit in taking conjugates.

To illustrate this, we again consider the antilinear map $\J$, but mapping into the conjugate $\overline{\CS}$. 
Define a map $j:\CS\to \overline{\CS}$ by
$j(\phi)=\overline{\J\phi}$, and this is a linear bimodule map, as we now show, for $a\in A$ and $\phi\in\CS$,
\begin{align*}
&j(a.\phi)=\overline{\J(a.\phi)}=\overline{\J a\J^{-1}\J\phi}=\overline{\J(\phi).a^*}=a.\overline{\J(\phi)}=a.j(\phi)\ ,\cr
&j(\phi.a)=j(\J a^* \J^{-1}\phi)=\overline{\J^2 a^*\J^{-1}\phi}=\epsilon\, \overline{ a^*\J^{-1}\phi}=
\epsilon\, \overline{ \J^{-1}\phi}.a= \overline{ \J\phi}.a=j(\phi).a\ .
\end{align*}
The other antilinear map we will need is the star operation on $\Omega$, extending the star operation on $A$. We define the bimodule map $\star:\Omega\to \overline{\Omega}$ by $\star\,\xi=\overline{\xi^*}$ for $\xi\in\Omega$. 

Recall next that given a left bimodule connection $(\CS,\nabla_\CS,\sigma_\CS)$ where $\sigma_\CS$ is invertible, we have a canonical left bimodule connection $\nabla_{\overline{\CS}}$ on $\overline{\CS}$ given by \cite{BegMa3}
\begin{align} \label{riconbar}
\nabla_{\overline{\CS}}(\overline{\phi})=(\star^{-1}\tens\id)\Upsilon\,\overline{\sigma_\CS{}^{-1}\nabla_\CS\phi}\ .
\end{align}

The condition for $j$ to intertwine the left connections is $(\id\tens j)\nabla_\CS=\nabla_{\overline{\CS}}\,j$, or 
\begin{align} \label{jpreserves}
(\id\tens j)\nabla_\CS\phi=\nabla_{\overline{\CS}}\,j(\phi)=\nabla_{\overline{\CS}}(\overline{\J\phi})=
(\star^{-1}\tens\id)\Upsilon\,\overline{\sigma_\CS{}^{-1}\nabla_\CS(\J\phi)}\ .
\end{align}
The difference $(\id\tens j)\nabla_\CS-\nabla_{\overline{\CS}}\,j$
 is a left module map, so to check the difference is zero, it is enough to do so on a set of left generators for $\CS$. Using (\ref{jpreserves}) we can calculate
\begin{align*}
\overline{\dirac\J\phi}=&\ \overline{\la\,\nabla_\CS (\J\phi)}=\overline{\la\,\sigma_\CS}\Upsilon^{-1}(\star\tens j)\nabla_\CS\phi\ ,\cr
\overline{\J \dirac\phi}=&\ j\,\dirac\phi = j\,(\la)\,\nabla_\CS\phi\ .
\end{align*}
To satisfy property (5), we need $ j\,(\la)=\epsilon'\, \overline{\la\,\sigma_\CS}\Upsilon^{-1}(\star\tens j)$, which we can restate, using $\xi\in\Omega^1$ as
\begin{align}  \label{lapreserves}
\J(\xi\la\phi)=\epsilon'\, (\la)\sigma_\CS(\J\phi\tens\xi^*)\ .
\end{align}
The reader may complain that we have used antilinear maps in the tensor product in (\ref{lapreserves}), but we have used them in  \textit{both} positions with a swap, which is legal. As long as we keep up the bookkeeping, all conjugates and antilinear maps stay legal.

We summarise the above results in the following theorem, stated in bimodule language. Note that we have not yet discussed the Hilbert space structure, we only refer to conditions (2)-(6). We denote by $\mathrm{bb}$ the canonical identification
$s\mapsto \overline{\overline{s}}$ of a bimodule $\CS$ with its double conjugate.

\begin{theorem} \label{sptripres}
Suppose that $\CS$ is an $A$-bimodule and
$j:\CS\to\overline{\CS}$ a bimodule map obeying $\overline{j}\,j=\epsilon\,\mathrm{bb}:
\CS\to\overline{\overline{\CS}}$. 
Suppose that   $(\CS,\nabla_\CS,\sigma_\CS)$ is a left bimodule connection, where $\sigma_\CS$ is invertible, and that $(\id\tens j)\nabla_\CS=\nabla_{\overline{\CS}}\,j$
for $\nabla_{\overline{\CS}}$ the canonical left connection on ${\overline{\CS}}$. 
Suppose that $\la:\Omega^1\tens_A \CS\to 
\CS$ is a bimodule map obeying $ j\,(\la)=\epsilon'\, \overline{\la\,\sigma_\CS}\Upsilon^{-1}(\star\tens j)$. Then $\dirac=\la\circ\nabla:\CS\to \CS$ and
$\J:\CS\to\CS$ defined by $j(\phi)=\overline{\J(\phi)}$  satisfy conditions (2)-(6)
for an odd spectral triple.

If there is a bimodule map $\gamma:\CS\to \CS$ with $\gamma^2=\id$, 
which intertwines the connection $\nabla_\CS$, and has
\[
\gamma\circ\la=-\la\circ(\id\tens\gamma):\Omega^1\tens_A \CS\to \CS\ ,\
j\circ\gamma=\epsilon''\,\overline{\gamma}\circ j:\CS\to\overline{\CS}\ ,
\]
then $(\dirac,\J,\gamma)$ satisfies the conditions (2)-(6) for an even spectral triple. 
\end{theorem}

\subsection{The complex valued inner product}\label{dirachermitian}
A hermitian inner product is antilinear in one position (in this case the first) and linear in the other, so it may be guessed that it appears rather more natural when we use conjugates. If we take a linear map $\<\!\<\ ,\ \>\!\>:\overline{\CS}\tens_A \CS\to \C$ then we have the right antilinearity properties, and \textit{explicitly writing the antilinear identity} we have, for $a\in A$ and $\phi,\psi\in\CS$,
\[
\<\!\< \overline{\psi} , a.\phi  \>\!\> = \<\!\< \overline{\psi}.a , \phi  \>\!\> = \<\!\< \overline{a^*.\psi} , \phi  \>\!\>\ ,
\]
which is the equation in property (1a), with explicit conjugates added. We have used the standard comma for inner product, but with the conjugate modules notation we could equally consistently have written $\<\!\< \overline{\psi} \tens \phi  \>\!\>$ instead of 
$\<\!\< \overline{\psi} , \phi  \>\!\>$. 

As we have an antilinear map $\J$, we can define a bilinear inner product, rather than a hermitian inner product, by
 $(\!(,)\!)=\<\!\<,\>\!\>\circ(j\tens\id):\CS\tens_A\CS\to \C$. This is now complex linear on both sides while the property of $\<\!\<\ ,\ \>\!\>$ under complex conjugation of the output appears now as $(\!(\psi,\phi)\!)^*=\eps(\!(\J\phi,\J \psi)\!)$. 
Then the isometry condition 
$\<\!\<\overline{\J\psi},\J\phi\>\!\>=\<\!\<\overline{\phi},\psi\>\!\>$ for $\J$ is now equivalent to $(\!(\psi,\J\phi)\!) = \epsilon\, (\!(\J\phi,\psi)\!)$, and this reduces to, for all $\phi,\psi\in\CS$,
\begin{align} 
(\!(\psi,\phi)\!) = \epsilon\, (\!(\phi,\psi)\!)\ .
\end{align}

Meanwhile, $\<\!\<\overline{\dirac\psi},\phi\>\!\>=\eps(\!(\J \dirac\psi,\phi)\!\)$ and $\<\!\<\overline{\psi},\dirac\phi\>\!\>=\eps(\!(\J \psi,\dirac\phi)\!\)$, so assuming (5) and relabelling $\psi$, we see that $\dirac$ being hermitian is equivalent to showing that
\begin{equation}\label{roundDherm} \eps'(\!( \dirac\psi,\phi)\!\)=(\!( \psi,\dirac\phi)\!\).\end{equation}
In the examples we shall deal directly with the definition of $\dirac$ being hermitian, but it is interesting to note some conditions on the bilinear inner product which would imply that $\dirac$ is hermitian. Note that bimodule connections extend canonically to tensor products.

\begin{proposition} For $\dirac$ constructed as in Theorem~\ref{sptripres}, suppose
\[
0=(\!( ,)\!\)\circ (\la\tens\id)\nabla_{\CS\tens\CS}:\CS\otimes_A\CS\to \C\ ,
\]
and also that
\[
(\!( ,)\!\)\circ ((\la)\sigma_\CS\tens\id)=- \, \epsilon'\, (\!( ,)\!\)\circ (\id\tens\la): \CS\otimes_A\Omega^1 \otimes_A \CS\to \C\ .
\]
Then $\dirac$ is hermitian.
\end{proposition}
\noindent\textbf{Proof:}\quad By definition of the connection on tensor products, the first equation is explicitly
\[
0=(\!( ,)\!\)\circ (\dirac\tens\id)+ (\!( ,)\!\)\circ ((\la)\sigma_\CS\tens\id)(\id\tens\nabla_\CS)\ ,
\]
and application of the second displayed equation gives (\ref{roundDherm}). \qquad$\largesquare$

Note that we do not require that $\<\!\<\ ,\ \>\!\>$ is the composition of a positive linear functional with an $A$-valued hermitian inner product and  $\nabla_S$ hermitian metric compatible (but both of these further features will apply in the $q$-sphere example). 

\subsection{Inner fluctuations}
Given a bimodule $L$, there is a functor $\mathcal{G}_L$ from the category ${}_A\mathcal{M}_A$ of $A$-bimodules to itself given by $\mathcal{G}_L(E)=\overline{L}\tens_A E\tens_A L$, sending a bimodule map $\theta:E\to F$ to
$\id\tens\theta\tens\id$. If we have a given isomorphism $L\tens_A \overline{L}\, \cong\, A$ of $A$-bimodules, then
\[
\mathcal{G}_L(E)\tens_A \mathcal{G}_L(F)=
\overline{L}\tens_A E\tens_A L \tens_A \overline{L}\tens_A F\tens_A L\, \cong\,\mathcal{G}_L(E\tens_A F)\ ,
\]
so the functor preserves the tensor product. 

The description of Morita contexts can be found in \cite{BassK} (and a $C^*$-algebra description in \cite{RieffelMorita}), and involves a bimodule $L$ so that the tensor product of $L$ with its dual, both ways round, is isomorphic to $A$, and the two isomorphisms obey associativity conditions. The special case we have is where the dual of the bimodule is its conjugate, and we get a non-degenerate inner product. In 
 \cite{NCline} this case is shown to give rise to an integer graded star algebra which is $L\tens_A \dots\tens_A L$ in positive degrees
 and $\overline{L}\tens_A \dots\tens_A \overline{L}$ in negative degrees. This star algebra can be thought of as the algebra of functions on a principal circle bundle on the noncommutative space, an idea defined formally in terms of a quantum principal bundle or Hopf-Galois extension \cite{BrzMaj:gau}. 
 
 Take $\overline{c}\in\overline{L}$ and $x\in L$ which are inverses under the product, i.e.\
$\overline{c}\tens x$ corresponds to $1\in A$. It will be convenient to write this identification as an inner product $\<,\>_L:\overline{L}\tens_A L\to A$. 
Then there is a linear map $\overline{c}\tens -\tens x:E\to \mathcal{G}_L(E)$ given by
$e\mapsto \overline{c}\tens e\tens x$, which is not necessarily a bimodule map. However it does have the tensorial property
\[
\xymatrix{
E\otimes_A F  \ar[rr]^{    (\overline{c}\tens -\tens x)\tens (\overline{c}\tens -\tens x) \ \ \ \ \ }  \ar[dr]^{\overline{c}\tens -\tens x }  & &  \mathcal{G}_L(E) \otimes_A \mathcal{G}_L(F)  \\
  &   \mathcal{G}_L(E\otimes_A F)  \ar[ur]^{\cong}
  }
\]

Now suppose we have a left bimodule connection $\nabla_L: L  \to \Omega^1\tens_A L$ and invertible
$\sigma_L: L\tens_A\Omega^1  \to \Omega^1\tens_A L$. 
If the connection preserves the inner product we get
\[
\nabla_{ \overline{L} }\tens\id+(\sigma_{ \overline{L} }\tens\id)(\id\tens\nabla_L)=\extd\circ \<,\>:\overline{L}\tens_A L\to \Omega^1\ ,
\]
and from this, remembering that the inner product is invertible, 
\[
\sigma_{ \overline{L} }{}^{-1}\,\nabla_{ \overline{L} }(\overline{c})\tens x=-\, \overline{c}\tens\nabla_L (x)\ .
\]
As $x$ is invertible, we can write $\nabla_L(x)=\kappa\tens x$ for some $\kappa\in\Omega^1$, and then we deduce 
$\sigma_{ \overline{L} }{}^{-1}\,\nabla_{ \overline{L} }(\overline{c})=-\,\overline{c}\tens\kappa$.

Now return to the assumptions and notations 
on the left bimodule connection $(\CS,\nabla_\CS,\sigma_\CS)$ and the action $\la:\Omega^1\tens_A \CS\to \CS$ which earlier we related to the Dirac operator.
Define an action of $\Omega^1$ on 
$\overline{L}\tens_A \CS\tens_A L$ by
\[
\xi\la(\overline{y}\tens \phi\tens x)=(\id\tens\la\tens\id)(\sigma_{\overline{L}}{}^{-1}(\xi\tens \overline{y})\tens \phi\tens x)
\]
and the Dirac operator $\dirac_{\mathcal{G}_L(\CS)}$ is the composition of this with the standard tensor product covariant derivative,
\[
\nabla_{\overline{L}}\tens\id\tens\id+(\sigma_{\overline{L}}\tens\id\tens\id)\big(
(\id\tens\nabla_\CS\tens\id)+(\id\tens\sigma_\CS\tens\id)(\id\tens\id\tens\nabla_L)\big)\ .
\]
On taking the composition we get some simplification, giving
\begin{align*}
\dirac_{\mathcal{G}_L(\CS)} =&\ (\sigma_{\overline{L}}{}^{-1}\nabla_{\overline{L}})\la\id\tens\id+\id\tens \dirac_\CS\tens\id+
(\id\tens(\la)\sigma_\CS\tens\id)(\id\tens\id\tens\nabla_L)\ .
\end{align*}
Now consider the commutative diagram, 
\[
\xymatrix{
\CS  \ar[r]^{\overline{c}\tens - \tens x\ \ \ }  \ar[d]^{\dirac_\CS }  &  \mathcal{G}_L(\CS)   \ar[d]^{  \dirac_{\mathcal{G}_L(\CS)} }  \\
  \CS  \ar[r]  &   \mathcal{G}_L(\CS)
  }
\]
where the bottom line is
\[
\phi \longmapsto \overline{x}\tens \dirac_\CS(\phi)\tens x- \overline{c}\tens \kappa\la \phi \tens x+ \overline{c}\tens (\la)\sigma_\CS(\phi\tens\kappa)\tens x\ .
\]
This can be rewritten as
\[
\xymatrix{
\CS  \ar[r]^{\overline{c}\tens - \tens x\ \ \ }  \ar[d]^{\dirac_\CS +\hat\kappa }  &  \mathcal{G}_L(\CS)   \ar[d]^{  \dirac_{\mathcal{G}_L(\CS)} }  \\
  \CS   \ar[r]^{\overline{c}\tens - \tens x\ \ \ }   &   \mathcal{G}_L(\CS)
  }
\]
where $\hat\kappa:\CS\to \CS$ for $\kappa\in\Omega^1$ is given by
\[
\hat\kappa(\phi)= (\la)\sigma_\CS(\phi\tens\kappa) - \kappa\la \phi\ .
\]
Rewriting (\ref{lapreserves}) gives $\J(\kappa^*  \la\J^{-1} \phi)=\epsilon'\, (\la)\sigma_\CS(\phi\tens\kappa)$, so 
\begin{align*}  
\hat\kappa(\phi)= \epsilon'\,      \J(\kappa^*  \la\J^{-1} \phi) -\kappa\la \phi  \ .
\end{align*}
If we follow \cite{ConMar} and specialise to the case where $L$ is $A$, then we can choose $x$ to be a unitary, in which case $\kappa^*=-\kappa$, and we have 
\begin{align*}  
\hat\kappa(\phi)= -\, \epsilon'\,      \J(\kappa  \la\J^{-1} \phi) -\kappa\la \phi  \ ,
\end{align*}
in agreement with the usual formula for inner fluctuations. In \cite{ConMar} it is explained that the inner fluctuations of the standard model of particle physics correspond to the gauge bosons other than the graviton, and arise via the mechanism of Morita equivalences.

\section{Examples of bimodule connections and Dirac operators}

Now we shall give three examples of our geometrical construction of Dirac operators from bimodule connections, on a matrix algebra, a quantum sphere, and a quantum disk. 

\subsection{A  Dirac operator on $M_2(\C)$} \label{SecMatrix}
Take the algebra $A=M_2(\C)$, with calculus
\[
\Omega^1=\Omega^{1,0}\oplus \Omega^{0,1}=M_2 \oplus M_2\ ,
\]
which we also write as $\Omega^{1,0}=M_2\,s$ and 
$\Omega^{1,0}=M_2\,t$, where $s=I_2\oplus 0$ and 
$t=0\oplus I_2$ are central elements (including $st=ts$), and to have a two dimensional calculus we impose $s^2=t^2=0$. 
The differential $\extd$ is the graded commutator $[E_{12}s+E_{21}t,-\}$ (i.e.\ the commutator when applied to even forms, and the anticommutator on odd forms). This is a star calculus, where we use the usual star on matrices and $s^*=-t$. 

Take an ansatz for a particular Dirac operator on the left module $\mathcal{S}=M_2(\C)\oplus M_2(\C)$, with the action of matrix product on each summand.
The Hilbert space inner product is $\<\!\<\overline{x\oplus u},y\oplus v\>\!\>=\mathrm{Tr}(x^*y+u^*v)$. 
Define $\dirac$ by the following formula,
\[
\dirac(x\oplus u)=(d_1u+uc_1) \oplus (d_2x+xc_2)\ ,
\]
for matrices $c_i,d_i$. Now
\begin{align*}
\<\!\<\overline{\dirac(x\oplus u)},y\oplus v\>\!\> =&\ 
\<\!\<\overline{(d_1u+uc_1) \oplus (d_2x+xc_2)},y\oplus v\>\!\> \cr
=&\ \mathrm{Tr}(u^*d_1^*y+c_1^*u^*y+x^*d_2^*v+c_2^*x^*v)\ ,\cr
\<\!\<\overline{x\oplus u},\dirac(y\oplus v)\>\!\> =&\ 
\<\!\<\overline{x\oplus u},(d_1v+vc_1) \oplus (d_2y+yc_2)\>\!\>  \cr
=&\ \mathrm{Tr}(u^*d_2y+u^*yc_2+x^*d_1v+x^*vc_1)\ .
\end{align*}
To have $\dirac$ hermitian we need $d_2=d_1^*$ and $c_2=c_1^*$. 
Define $\J:\mathcal{S} \to \mathcal{S}$ by $\J(x\oplus u)=(-u^*)\oplus x^*$, so $\epsilon=-1$. Now we have $
\J b\J^{-1}(x\oplus u)=x\,b^*\oplus u\,b^*$. 
Next
\begin{align*}
\J \dirac(x\oplus u)=&\ (-(d_1^*x+xc_1^*)^*) \oplus (d_1u+uc_1)^*\ ,\cr
\dirac\J(x\oplus u)=&\ (d_1x^*+x^*c_1) \oplus (-d_1^*u^*-u^*c_1^*)
\end{align*}
so $c_1=-d_1$ gives $\epsilon'=1$. Now the grading operator $\gamma(x\oplus u)=(-x)\oplus u$ completes the set of operators
for dimension $n=2$ with $\eps''=-1$. 
We calculate
\[
[\dirac,a](x\oplus u)=[d_1,a]u \oplus [d_1^*,a]x\ .
\]
To fit with the differential structure we set $d_1=E_{12}$, and seek $\la$ so that 
$\extd a\la (x\oplus u)=[\dirac,a](x\oplus u)$ or 
\[
([E_{12},a]\oplus [E_{21},a])\la 
(x\oplus u)=[E_{12},a]u \oplus [E_{21},a]x
\]
which we solve by defining the action of $\Omega^1$ as  $(p\oplus q)\la (x\oplus u)=p\,u\oplus q\,x$. 
The required connection $\nabla_\CS$ is then 
\[
\nabla_\CS(x\oplus u) = \extd x\tens (1\oplus 0) + 
\extd u\tens (0\oplus 1)\ ,
\]
and a little calculation gives, for $\xi\in\Omega^1=M_2\oplus M_2$,
\[
\sigma_\CS((x\oplus u)\tens\xi) = x.\xi\tens (1\oplus 0) + 
 u.\xi\tens (0\oplus 1)\ .
\]
With these choices we then verify the condition in Proposition~\ref{Diracnabla},
\begin{align*}
\J[\dirac,b]\J^{-1}(x\oplus u) =&\ u[E_{12},b^*]\oplus x [E_{21},b^*]
=\la\sigma\big((x\oplus u) \tens ([E_{12},b^*]\oplus  [E_{21},b^*])\big)\ .
\end{align*}
so this proposition applies. Similarly,  we can check directly that 
\[
\<\!\<\overline{ \J(x\oplus u)  },\J(y\oplus v)\>\!\>=
\<\!\<\overline{ (-v^*)\oplus y^*  },(-u^*)\oplus x^*\>\!\>=
\mathrm{Tr}(vu^*+yx^*) = \<\!\<\overline{x\oplus u},y\oplus v\>\!\>
\]
so $\J$ is an isometry. We can also recover this and that $\dirac$ is hermitian from our deduced data and application of Section~\ref{dirachermitian}.

To compare this with the known classification of spectral triples on matrix algebras, we refer to \cite{Barrmatrix}. We write $\CS=\C^2\tens M_2(\C)$ by writing $x\oplus u$ as a vector, where $x,u\in M_2(\C)$. We take the signature $(0,2)$ Clifford algebra with $\{\gamma^i,\gamma^j\}=-2\delta_{ij}$ given by $\gamma^i=\imath\sigma^i$ in terms of Pauli matrices, $i=1,2$. We set $\gamma=\imath^{3}\gamma^1\gamma^2=-\sigma^3$ which agrees with the one above. We also need an antilinear $C$ such that $C^2=\eps=-1$, $(Cv, Cw)=(w,v)$ for the standard left-antilinear inner product on $\C^2$, and $C\gamma^i=\eps'\gamma^iC=\gamma^iC$. The operation 
\[ C\begin{pmatrix} v_1 \cr v_2\end{pmatrix}=\begin{pmatrix} -\overline{v_2} \cr \overline{v_1}\end{pmatrix}\]
does the job and $\J=C\tens(\ )^*$ then gives the same $\J$ as above. Finally, our Dirac operator can now be written as 
\[ \dirac=-{1\over 2}\left(\gamma^1\tens[\gamma^1,\ ]-\gamma^2\tens[\gamma^2,\ ]\right)\]
which is a specific member of the general class of spectral triple here (where commutators in general are by arbitrary antihermitian matrices). We have seen how this arises naturally from an action $\la$ and a bimodule connection.

\subsection{A Dirac operator on the noncommutative Hopf fibration} \label{ncHopfDirac}

We follow the construction of the $q$-Dirac operator on the standard $q$-sphere as a framed quantum homogeneous space in \cite{Ma:rieq} but with a  couple of constant parameters (which can be seen as normalisations) and now with consideration of $*$, $\J$ and an inner product which we not covered there. We recall that the algebra $\C_q[SU_2]$ has generators $a,b,c,d$, which are assigned grades $|a|=|c|=1$ and $|b|=|d|=-1$ and we use the 
conventions where $ba=qab$ etc. The standard $q$-sphere 
$A=\C_q[S^2]$ is the subalgebra of grade zero elements in $\C_q[SU_2]$. The usual 3D calculus for $\C_q[SU_2]$ in \cite{Wor} has basis 1-forms $e^0,e^\pm$ of grades $|e^0|=0$ and $|e^\pm|=\pm 2$ and bimodule commutation relations $e^0x=q^{2|x|}x e^0$ central and $e^\pm x=q^{ |x|}x e^\pm$. For a calculus on the sphere, we take the horizontal forms (with basis $e^\pm$ of degree $|e^\pm|=\pm 2$), and then the grade zero submodule. This means that $\Omega^{1,0}$ and $\Omega^{0,1}$ for the cotangent bundle on the $q$-sphere can be identified with the degree $\mp 2$ subspaces of $\C_q[SU_2]$ respectively. Note also in this construction that both $\Omega^1$ and the horizontal forms $\Omega^1_{hor}$ on $\C_q[SU_2]$ are free modules (with basis $e^\pm,e^0$ and $e^\pm$ respectively) so we have a canonical projection $\pi:\Omega^1\to \Omega^1_{hor}$ of free left $\C_q[SU_2]$ which will be useful in computations, given by $e^0\to 0$. This is the set-up for the quantum Riemannian geometry of the standard $q$-sphere as a quantum homogeneous space from the quantum Hopf fibration \cite{Ma:rieq}; the $q$-monopole connection on the quantum principal bundle induces a canonical choice of `quantum Levi-Civita' connection on $\Omega^1=\Omega^{1,0}\oplus \Omega^{0,1}$.

For the spin bundle we similarly set generators $f^\pm$ with grades $|f^\pm|=\pm1$, and
$\S_\pm$ to be the grade zero elements in $\C_q[SU_2].f^\pm$, with $\CS=\CS_+\oplus \CS_-$.
Suppose that the generators commute with all grade zero algebra elements.
 In other words, $\CS_\pm$ can be identified with the degree $\mp1$ subspace of $\C_q[SU_2]$ and as a (bi)-module over $\C_q[S^2]$ (which means that $f^\pm$ commute with elements of $A$). This is again the set-up used in \cite{Ma:rieq} for the spin bundle as charge $\pm1$ $q$-monopole sections and again the $q$-monopole induces a covariant derivative $\nabla_S:\CS\to \Omega^1\tens_A\CS$. This is well-known and given explicitly by
\[
\nabla_\S(x\,f^++y\,f^-)=\pi\extd x.a \tens    d.f^+ -q^{-1}\, \pi\extd x. c\tens   b .f^+ +\pi \extd y. d \tens   a .f^- - q\,\pi\extd y.b \tens   c .f^-\ 
\]
One can check that this is a bimodule connection with
\[
\sigma_S((x f^++ yf^-)\tens f e^{\pm})=xf e^{\pm }(a\tens d - q^{-1}c\tens b)f^+ + y f e^{\pm }(d\tens a - q b\tens c)f^-\]
for $f$ of degree $\mp 2$.

For the action $\la$ of $\Omega^1$ on the spinors which preserves grades,  we follow \cite{Ma:rieq} and  set
\[
 f e^+\la y f^- = \alpha\, f y f^+\ ,\ f e^-\la x f^+ = \beta\, f x f^- , \quad \la:\Omega^{1,0}\tens\CS_-\to \CS_+,\quad \la:\Omega^{0,1}\tens\CS_+\to \CS_-\]
and other degree combinations zero, where we have explicitly inserted two constant complex parameters $\alpha,\beta$ (one could absorb one of these in the normalisation of the $f^\pm$). Apart from the constant parameters, this is just the product of the appropriate degree subspaces inside $\C_q[SU_2]$ as in \cite{Ma:rieq}. 

If we write $\pi\extd x=\del_+ x\,e^++\del_- x\,e^-$, the Dirac operator $\dirac=(\la\tens\id)\nabla_\S$ comes out  for $|x|=-1$ and $|y|=1$ as
\[
\dirac(x\,f^++y\,f^-) = \alpha q^{-1} \del_+ y\, f^+ +\beta q\, \del_- x\,f^- 
\]
which apart from the $\alpha,\beta$ weightings completes our recap of the Dirac operator introduced in \cite{Ma:rieq}. The grading bimodule map is given by $\gamma=\pm\id$ on $\CS_\pm$.  

The new ingredient we need beyond \cite{Ma:rieq} is $\CJ$. We set $\J(x f^\pm)=\pm\delta^{\pm1}\,x^* f^\mp$ for $\delta$ real, giving $\eps=-1$ and  $\epsilon''=-1$. The connection preserves $j$ since it vanishes on the generators, while using  $e^{\pm *}=-q^{\mp 1}e^\mp$ we get
\begin{align*}
&(\la)\sigma_\CS(\J (x\, f^+)\tens (y\, e^{-})^*) = -\,\delta\,q (\la)\sigma_\CS(x^* f^- \tens e^+\,y^*)\cr
 =&\  -\,\delta\,q^{-1} (\la)\sigma_\CS(x^* f^- \tens y^*\,e^+) = -\,\delta\,q^{-1} (\la)(x^*\,y^*\,  e^{+ }(d\tens a - q b\tens c)f^-)\cr
 =&\  -\,\delta\,q^{-2} (\la)(x^*\,y^*(d\,  e^{+ }\tens a - q\, b\,  e^{+ }\tens c)f^-)
 =  -\,\alpha\,\delta\,q^{-2} x^*\,y^*\,f^+\ ,\cr
&(\la)\sigma_\CS(\J (x\, f^-)\tens (y\, e^{+})^*) = \delta^{-1}q^{-1}\,(\la)\sigma_\CS(x^* f^+  \tens e^-\,y^*)\cr
 =&\  \delta^{-1}q^{1}\,(\la)\sigma_\CS(x^* f^+  \tens y^*\,e^-) = \delta^{-1}q^{1}\,(\la)(x^*\,y^*\, e^{- }(a\tens d - q^{-1}c\tens b)f^+)\cr
 =&\  \delta^{-1}q^{2}\,(\la)(x^*\,y^*(a\, e^{- }\tens d - q^{-1}c\, e^{- }\tens b)f^+)
 = \beta\, \delta^{-1}q^{2}\,x^*\,y^*\, f^-\ .
\end{align*}
Referring back to (\ref{lapreserves}) with $\eps'=1$, we need to compare these results with
\begin{align*}
\J(y\, e^{-}\la x\, f^+) =&\ \J(\beta\,y\,x\,f^-) = -\,\delta^{-1}\,\beta^*\, (y\,x)^*\, f^+\ ,\cr
\J(y\, e^{+}\la x\, f^-) =&\ \J(\alpha\,y\,x\,f^+) = \delta\,\alpha^*(y\,x)^*\,f^-\ .
\end{align*}
In the case $\epsilon'=1$, (\ref{lapreserves}) becomes the condition 
\begin{equation}\label{qsphereparam} \delta^2\,\alpha^* = \beta\, q^{2},\end{equation}
 for $q$ real, which requires that $\beta/\alpha^*$ is real. Assuming the latter, we therefore define $\delta$ as the (say, positive) square root of $\beta q^2/\alpha^*$ and have  now satisfied all the algebraic  axioms (2)-(6) of a spectral with dimension $n=2$, by Theorem~\ref{sptripres}. 
 
Next we define a positive hermitian inner product $\<,\>:\overline{ \S}\tens_A \S \to A$ by the following, for some $\mu>0$,
\[
\<\overline{ x_+\, f^+ + x_-\, f^-}, y_+\, f^+  +  y_-\, f^-  \>=x_+{}^*\,y_+ + \mu\, x_-{}^*\,y_- \ .
\]
So far we have a $A$ valued inner product, but we really need an honest $\C$ valued inner product for a Dirac operator. We define
\[
\<\!\<,\>\!\>=\frac{\smallint \<,\>\,e^+\wedge e^-}{\smallint e^+\wedge e^-}\ .
\]
where $\smallint$ is the de Rham cohomology class in 
$H_{dR}^2(\C_q[S^2])\cong \C$. This gives a hermitian inner product
$\<\!\<,\>\!\>:\overline{ \S}\tens_A \S \to \C$. (This is just the Haar integral of the $A$-valued inner product.)

\begin{proposition} \label{prrp1} For $ \<\!\<\ ,\ \>\!\>$ defined by the Haar integral and $\mu=q\delta^{-2}$, $\dirac$ is hermitian. 
\end{proposition}
\noindent\textbf{Proof:}\quad 
We have, using our notations,
\begin{align*}
\<\overline{\dirac(x\,f^+)},y\,f^-\> =&\ \beta^*\, q\, \<\overline{    \del_- x\,f^-    },y\,f^-\>  =\
 \beta^*\, q\,\mu\,     (\del_- x)^*\,y \ ,\cr
\<\overline{x\,f^+},\dirac(y\,f^-)\> =&\ \alpha\, q^{-1}\, \<\overline{x\,f^+},    \del_+ y\, f^+  \>
= \alpha\, q^{-1}\, x^*\, \del_+ y
\ 
\end{align*}
for all $x\,f^+$ and $y\,f^-$ of grade zero. 
So if $ \beta^*\, q\,\mu=\alpha$ we have
\begin{align*}
\<\overline{x\,f^+},\dirac(y\,f^-)\>  - \<\overline{\dirac(x\,f^+)},y\,f^-\>
= \alpha\,q^{-1}\,(x^*\del_+ y 
- q\,(\del_- x){}^*y)\ .
\end{align*}
Using $|x|=-1$ and $|y|=1$, with $\pi\extd x=\del_+ x\,e^++\del_- x\,e^-$ etc, we also have
\begin{align*}
\pi\extd(x^*y)=&\ (x^*\del_+y - q\, (\del_- x){}^*y)\, e^+ + (x^*\del_-y - q^3\, (\del_+x){}^*y)\, e^-\ ,\cr
\pi\extd(x^*y\, e^-)=&\ 
\pi\extd(x^*y)\wedge e^-= (x^*\del_+y - q\, (\del_- x){}^*y)\, e^+ \wedge e^-\ ,
\end{align*}
which gives 
\begin{align*}
\<\!\<\overline{x\,f^+},\dirac(y\,f^-)\>\!\>  = \<\!\<\overline{\dirac(x\,f^+)},y\,f^-\>\!\>\ ,
\end{align*}
on applying the cohomology class given by the Haar integral. Taking the complex conjugate provides the other equation needed to show that $\dirac$ is hermitian. The condition on the parameters here is equivalent to the one stated given that we already assumed (\ref{qsphereparam}). \qquad$\largesquare$

Proceeding with $\dirac$ hermitian by the above proposition, it remains to look at the isometry property of $\J$. For this we note that the underlying  Haar integral on functions, $\int:\C_q[SU_2]\to \C$, is well-known to be a twisted trace in that there is an algebra automorphism $\varsigma$ such that $\int xy = \int \varsigma(y)x$ for all $x,y$ in $\C_q[SU_2]$. Explicitly, 
\[ \varsigma(a^ib^jc^kd^l)=q^{2(l-i)}a^ib^jc^kd^l\]
on monomials from which one can see that $\varsigma$ preserves degree and skew-commutes with $*$ in the sense $\varsigma(x^*)=(\varsigma^{-1}(x))^*$ for all $x\in \C_q[SU_2]$.

\begin{proposition} \label{prrp2} For $ \<\!\<\ ,\ \>\!\>$ defined by the Haar integral and $q\ne 1$, $\J$ is not an isometry but obeys
\[ \<\!\<\overline{\J(x.f^\pm)},\J(y.f^\pm)\>\!\>=q^{\pm 1} \<\!\<\overline{\varsigma^{-1}(y).f^\pm},x.f^\pm\> \!\>,\quad\forall |x|,|y|=\mp1\]
where $\varsigma$ is an algebra automorphism whereby the Haar integral is a twisted trace.
\end{proposition}
\noindent\textbf{Proof:}\quad 
Consider, for $|x|=|y|=\mp1$,
\[
\<\overline{\J(x.f^\pm)},\J(y.f^\pm)\> =  \<\overline{  \pm\delta^{\pm1}\,x^* f^\mp  },  \pm\delta^{\pm1}\,y^* f^\mp  \> 
= \left\{\begin{array}{cc}  \delta^2\,\mu\,x\,y^* & \mathrm{upper\ sign} \\ \delta^{-2}\, x\,y^* & \mathrm{lower\ sign}\end{array}\right.
\]
and compare this to 
\[
\<\overline{y.f^\pm},x.f^\pm\> = \left\{\begin{array}{cc}y^*\,x & \mathrm{upper\ sign} \\  \mu\,y^*\,x & \mathrm{lower\ sign}\end{array}\right.
\]
Given than $\delta^2\mu=q$ and integrating, the twisted trace property tells us that $\J$ is some kind of twisted $q$-isometry in the manner stated. It is clear when $q\ne 1$ that we can fun instances proving that $\CJ$ is not a usual isometry.
\qquad$\largesquare$

Note that the Haar integral is not the only choice. If we take an ordinary trace in the form of a linear map $\tau:A\to \C$ such that $\tau(xy)=\tau(yx)$, we can define
$\<\!\<\ ,\ \>\!\>=\tau\<\ ,\ \>$ in the same way, the above proof shows that we do then have $\CJ$ an isometry when $\mu=\delta^{-2}$, but we would then lose that $\dirac$ is hermitian as this depended on translation invariance in the form of vanishing on a total differential.

\medskip
To summarise, to finish off the algebraic conditions (2)-(6)  we needed $  \alpha = \beta^*\, q^{2}\,\delta^{-2}  $ and for $\dirac$ to be hermitian we need $\alpha= \beta^*\, q\,\mu$ which, given the first condition is equivalent to $\mu=q\delta^{-2}$. 
However we cannot in general make $\J$ an isometry when $q\ne 1$. Finally, by rescaling of the $f^\pm$ while preserving the form of our other constructions (this requires $f^+$ to change by at most a phase), we can without loss of generality set $\alpha=1$ and then have only one free parameter $\beta>0$ in our above construction, with $\delta=\sqrt{\beta} q$ and $\mu=\beta^{-1} q^{-1}$ uniquely determined up to the sign of $\delta$. Thus we have a 1-parameter moduli of Dirac operators under the above construction, with $\beta=1$ recovering the $\dirac$ introduced in \cite{Ma:rieq}. 

Considering the results on Dirac operators on the non commutative sphere in \cite{DLPS-sphere}, this should not come as a surprise that we cannot obey all the conditions. There the conclusion was to sacrifice the bimodule condition, which is (4) on our list, and replace it by the commutator being a compact operator. However we have kept all the algebraic conditions including the bimodule condition and  $\nabla_\CS$ preserving $j$,  kept $\dirac$ being hermitian and dropped only that $\CJ$ is an isometry in favour of some twisted $q$-version of that.

\subsection{A Dirac operator on the quantum disk}  \label{diskdir}
There is an algebra of functions on a `deformed disk' 
 $A=\C_q[D]$,  \index{deformed disk, $\C_q[D]$} generated by $z$ and $\bar z$
with commutation relation $z\bar z=q^{-2}\bar zz-q^{-2}+1$ and involution $z^*=\bar z$ with $q$ real and nonzero  (see \cite{klilesdisk}). 
The algebra  $\C_q[D]$ is $\mathbb{Z}$ graded, by $|z|=1$ and $|\bar z|=-1$. 
 Putting $\Y =1-\bar zz$, we have $z\Y =q^{-2}\Y z$ and $\bar z\Y =q^{2}\Y \bar z$, so for any polynomial $p(\Y )$
\[
z.p(\Y )=p(q^{-2}\Y ).z\ ,\quad \bar z.p(\Y )=p(q^{2}\Y ).\bar z\ .
\]

There is a differential calculus given by
\begin{eqnarray*}
\extd z\wedge \extd \bar z=-q^{-2}\, \extd \bar z\wedge \extd z\ &,& z. \extd z=q^{-2}\, \extd z. z\ ,\quad 
z. \extd \bar z=q^{-2}\, \extd \bar z. z\ ,\cr  \extd z\wedge \extd z=\extd \bar z\wedge \extd \bar z=0  &,&
\bar z.  \extd z=q^{2}\, \extd z. \bar z\ ,\quad \bar z.  \extd \bar z=q^{2}\, \extd \bar
z.  \bar z\ .
\end{eqnarray*} \label{dpolydisk}
A proof by induction on powers of $\Y $ gives, for any polynomial $p(\Y )$,
\begin{align} \label{missedp}
\extd p(\Y ) =&\ q^2  \frac{p(q^{-2}\Y )-p(\Y )}{\Y (1-q^{-2})} z\, \extd \bar z\ + 
 \frac{p(q^{2}\Y )-p(\Y )}{\Y (1-q^{2})}\bar z\,  \extd z\ .
\end{align}

Recall that the $*$-Hopf algebra $U_q(su_{1,1})$ is defined by generators $X_+,X_-$ and an invertible grouplike generator $q^{H\over 2}$ with 
\[
 q^{H\over 2}X_\pm q^{-{H\over 2}}=q^{\pm 1}X_\pm,\quad [X_+,X_-]={q^H-q^{-H}\over q-q^{-1}},\quad \Delta X_\pm=X_\pm\tens q^{H\over 2}+
q^{-{H\over 2}}\tens X_\pm   \]
and  $*$-structure $X_+^*= -X_-$,  $(q^{H\over 2})^*=q^{H\over 2}$ (we follow the conventions of \cite{Ma:book}). 
There is a left action of $U_q(su_{1,1})$ on $\C_q[D]$ (similar to that given by \cite{klilesdisk}, but adjusted to be unitary in the sense of
\cite{Ma:book}, i.e. $(h \la a)^*= S(h^*)\la a^* $) given by
\begin{align*}
&X_{\pm}\la 1=0\ ,\ q^{H\over 2}\la 1=1\ ,\ q^{H\over 2}\la z=q^{-1}z\ ,\ q^{H\over 2}\la \bar z=q\,\bar z\ ,\cr
&X_+\la  z = q^{-1/2}  \ , 
X_+\la\bar z= -q^{-1/2} \bar z^2
\ ,\ 
X_- \la \bar z= q^{1/2}
\ , 
X_- \la z= -q^{1/2} z^2 \ .
\end{align*}
This action extends to the calculus by
\begin{align*}
&q^{H\over 2}\la \extd z=q^{-1}\extd z\ ,\ q^{H\over 2}\la \extd\bar z=q\,\extd\bar z\ ,\ X_+\la \extd z=0\ ,\ X_-\la \extd\bar z=0\ ,\cr
&X_+\la \extd\bar z=-q^{-1/2}(\extd\bar z\,\bar z+\bar z\,\extd\bar z)\ ,\ X_-\la z=-q^{1/2}(z\,\extd z+\extd z\,z)\ .
\end{align*}

Now we consider integration on the deformed disk. In \cite{klilesdisk} an integral is given which has classical limit the Lebesgue integral on the unit disk, and is subsequently used to examine noncommutative function theory on the disk. However we use another integral with $U_q(su_{1,1})$ invariance.
A partially defined map $\int:\C_q[D]\to \C$, invariant for the $U_q(su_{1,1})$ action, is defined by
\[
\int  \Y^{n+1} =\frac{1}{[n]_{q^{-2}}} \ ,\quad n\ge 1\ ,
\]
and $\int$ applied to any monomial of nonzero grade gives zero. (We shall not go into detail over the domain of this integral.) To spell the invariance out explicitly, we require that the following diagram commutes:
\[
\xymatrix{
 U_q(su_{1,1}) \tens \C_q[D]  \ar[r]^{\ \ \ \quad  \la }  \ar[d]^{\id\tens\int }  &  \C_q[D]  \ar[d]^\int  \\
   U_q(su_{1,1}) \tens \C  \ar[r]^{\ \quad \epsilon\tens\id}   &  \C
  }
\]
We take care that there are two conflicting views of what is going on with the quantum disk. As a unital $C^*$ algebra, $\C_q[D]$ corresponds to a deformation of a compact topological space, the closed unit disk. Classically this is not a manifold, but it is a manifold with boundary. However the $U_q(sl_2)$ action is taking us in quite a different direction - classically it corresponds to the M\"obius action on the open disk, and as such its invariants are really related to hyperbolic space, rather than the closed unit disk. The problem with the integral is simply that the classical volume of hyperbolic space, under its usual invariant measure, is infinite.
Recalling that $\Y =1-\bar zz$, we have
\begin{align*}
q^{1/2}X_+\la \Y =&\ q^{-1}\bar z^2z-q^{-1}\bar z= - q^{-1}\bar z\Y \ ,\cr
q^{-1/2}X_-\la \Y =&\ - q^{-1}z+q^{-1}\bar zz^2=-q^{-1}\Y z\ ,\cr
q^{H\over 2}\la \Y =&\ 1-(q^{H\over 2}\la\bar z)( q^{H\over 2}\la z)=\Y \ .
\end{align*}
and induction gives, 
\[
q^{1/2}X_+\la \Y^n=-q^{-1}\bar z[n]_{q^{-2}}\Y^n\ ,\quad n\ge 0\ .
\]
Then
\begin{align*}
q^{1/2}X_+\la(z\Y^n)=&\ \Y^n-z\bar z[n]_{q^{-2}}\Y^n  \cr
=&\ -\ q^{-2}  [n-1]_{q^{-2}} \Y^n+q^{-2}   [n]_{q^{-2}}   \Y^{n+1}  \ .
\end{align*}
By invariance, the integral applied to this should give zero, so we get
\[
\int  [n-1]_{q^{-2}} \Y^n=\int   [n]_{q^{-2}}   \Y^{n+1} \ ,
\]
which, on choosing a normalisation, gives the formula we gave for the integral. 

We next show that the integral is a twisted trace, in the sense
\[
\int a\, b=\int \varsigma(b)\,a\ ,\quad \forall a,b \in \C_q[D]
\]
for the degree algebra automorphism $\varsigma(b)=q^{2|b|}\,b$ on homogeneous elements. This can be shown for $b=z$ by $\int az=0$ unless $a$ has degree $-1$, and in the $-1$ case writing 
$a=\bar z\,p(w)$ for a polynomial $p$. The result is then checked by explicit calculation. A straightforward inductive argument then extends this result to  $b=z^m$ and similarly for negative degrees. 

Next we take generators $\{s,\bar s\}$ of a spinor bimodule $\mathcal{S}$, with relations for homogenous $a\in \C_q[D]$,
\[
s.a=q^{|a|}\,a.s\ ,\quad \bar s.a=q^{|a|}\,a.\bar s\ 
\]
where the power of $q$ in the commutation relations is half that for the relations with $\extd z,
\bar\extd z$. Suppose that $\nabla_\mathcal{S}(s)=0$ and $\nabla_\mathcal{S}(\bar s)=0$, then 
define 
\[
\extd z\la \bar s=\alpha\,\Y \,s\ ,\ \extd\bar z\la  s=\beta\,\Y \,\bar s\ ,\ 
\extd z\la  s=0\ ,\ \extd\bar z\la  \bar s=0\ 
\]
for two parameters $\alpha,\beta$. We have
\begin{align*}
\sigma_\mathcal{S}(s\tens\extd a) =&\ \nabla_\mathcal{S}(s.a)-\nabla_\mathcal{S}(s).a
=  q^{|a|} \nabla_\mathcal{S}(a.s)-\nabla_\mathcal{S}(s).a = q^{|a|} \extd a\tens s\ ,\cr
\sigma_\mathcal{S}(\bar s\tens\extd a) =&\ q^{|a|} \extd a\tens \bar s\ .
\end{align*}
Then the Dirac operator is, writing $\extd a=\tfrac{\partial a}{\partial z} \extd z+\tfrac{\partial a}{\partial \bar z} \extd \bar z$, 
\begin{align*}
\dirac(a.s)=&\ (\tfrac{\partial a}{\partial z}\extd z+\tfrac{\partial a}{\partial \bar z}\extd \bar z)\la s=\beta\,(\tfrac{\partial a}{\partial \bar z})\,\Y \,\bar s\ ,\cr
\dirac(a.\bar s)=&\ (\tfrac{\partial a}{\partial z}\extd z+\tfrac{\partial a}{\partial \bar z}\extd \bar z)\la \bar s=\alpha\,(\tfrac{\partial a}{\partial z})\,\Y \,s\ .
\end{align*}
We set $\gamma(s)=s$ and $\gamma(\bar s)=-\bar s$. 

Next we set $\mathcal{J}(s)=\delta\bar s$ and $\mathcal{J}(\bar s)=-\delta^{-1} s$ for some real parameter $\delta$,  giving $\eps=-1$ and $\eps''=-1$. Since $\nabla_S$ on the generators is zero, it is clear that 
the connection preserves $j$. We compute using the definitions and commutation rules above that
\begin{align*} (\la)\sigma_S(\J(a s)\tens \extd z)=& \delta(\la) \sigma_S(\bar s a^*\tens \extd z) = \delta q q^{|a^*|}  (\la)(a^*\extd z\tens \bar s)=\delta q q^{|a^*|}  a^*\alpha\Y s\\
\J(\extd \bar z\la a s)=& q^{2|a|}\J( a\extd\bar z\la s)=q^{2|a|}\J( \beta a \Y \bar s)=\beta^* \J( \Y  a \bar s)=\beta^* (-\delta^{-1})s (\Y a)^*\\
=&-\beta^*\delta^{-1}q^{|a^*|}a^*\Y s\ 
\end{align*}
for all $a$ of homogeneous degree. Similarly for the other cases. Hence (\ref{lapreserves}) holds with $\eps'=1$ provided 
\begin{equation}\label{qdiskparamJ} \delta^2 q\alpha=-\beta^*.\end{equation}
We assume this and Theorem~\ref{sptripres} then tells us that (2)-(6) hold  with  dimension $n=2$. 

Finally, we 
define a positive hermitian inner product $\<,\>:\overline{ \S}\tens_A \S \to A$ by the following, for some $\mu>0$,
\[
\<\overline{ s\,a_++\bar s\, a_-}, s\,b_++ \bar s\,b_-  \>=a_+{}^*\,\Y \,b_+ + \mu\, a_-{}^*\,\Y \,b_- \ 
\]
for all $a_\pm,b_\pm\in A$. Now we have, 
\begin{align*}
\<\overline{ \dirac(a.\bar s)}, b. s  \> =&\ \<\overline{ \alpha\,\tfrac{\partial a}{\partial z}\,\Y \,s  }, b. s  \> =\alpha^*q^{-|b|-|\tfrac{\partial a}{\partial z}|}
\<\overline{ s\,\tfrac{\partial a}{\partial z}\,\Y   },  s.b  \>  \cr
=&\ \alpha^*q^{-|b|-|\tfrac{\partial a}{\partial z}|}  \Y \, \tfrac{\partial a}{\partial z}{}^*\, \Y     \, b
= \alpha^*q^{|b|-|\tfrac{\partial a}{\partial z}|}  \Y \, \tfrac{\partial a}{\partial z}{}^*b\, \Y     \ ,\cr
\<\overline{ a.\bar s}, \dirac(b. s)  \> =&\ \<\overline{ a.\bar s}, \beta\,\tfrac{\partial b}{\partial \bar z}\,\Y \,\bar s  \> = \beta\, q^{ -|a|-|\frac{\partial b}{\partial \bar z}|  }
\<\overline{ \bar s.a}, \bar s \, \tfrac{\partial b}{\partial \bar z}\,\Y   \> \cr
=&\ \mu\, \beta\, q^{ -|a|-|\frac{\partial b}{\partial \bar z}|  }a^*\,  \Y     \, \tfrac{\partial b}{\partial \bar z}\,\Y  
= \mu\, \beta\, q^{ |a|-|\frac{\partial b}{\partial \bar z}|  } \Y     \, a^*\, \tfrac{\partial b}{\partial \bar z}\,\Y  \ .
\end{align*}
If we define the complex valued inner product $\<\!\<,\>\!\>$ by the integral of $\<,\>$, then as the integral vanishes unless the grade is zero we find
\begin{align*}
\<\!\<\overline{ \dirac(a.\bar s)}, b. s  \>\!\>
=&\ \alpha^*\int \Y \, \tfrac{\partial a}{\partial z}{}^*b\, \Y     \ ,\cr
\<\!\<\overline{ a.\bar s}, \dirac(b. s)  \>\!\> 
=&\ \mu\, \beta \int \Y     \, a^*\, \tfrac{\partial b}{\partial \bar z}\,\Y  \ .
\end{align*}
Now
\begin{align*}
\extd(a^*b)=&\ (\tfrac{\partial a}{\partial z}\,\extd z+\tfrac{\partial a}{\partial \bar z}\, \extd \bar z)^*b+a^*(\tfrac{\partial b}{\partial z}\,\extd z+\tfrac{\partial b}{\partial \bar z}\,\extd \bar z)=
(\extd\bar z\,\tfrac{\partial a}{\partial z}{}^*+\extd  z\,\tfrac{\partial a}{\partial \bar z}{}^*)b+a^*(\tfrac{\partial b}{\partial z}\,\extd z+\tfrac{\partial b}{\partial \bar z} \, \extd \bar z)\ ,
\end{align*}
so we get (using the previous notation) $\tfrac{\partial a^*b}{\partial \bar z}=q^{2(|b|-|\tfrac{\partial a}{\partial z}|)}\tfrac{\partial a}{\partial z}{}^*b+a^*\tfrac{\partial b}{\partial \bar z}$, so if $\mu\, \beta=-\alpha^*$ which, given (\ref{qdiskparamJ}) is equivalent to
\begin{equation}\label{qdiskparam}\mu =q^{-1}\delta^{-2}, \end{equation}
 we get
\[
\<\!\<\overline{ \dirac(a.\bar s)}, b. s  \>\!\> - \<\!\<\overline{ a.\bar s}, \dirac(b. s)  \>\!\> =
\alpha^*\int \Y \,  \tfrac{\partial a^*b}{\partial \bar z}\, \Y   \ .
\]
With this choice, to show that $\dirac$ is hermitian all we need to show is that for all $a$ with $|a|=-1$ we have
\[
\int \Y \, \tfrac{\partial a}{\partial \bar z}\, \Y =0\ .
\]
We set $a=\bar z\Y^m$ for some $m\ge 1$, and then $\frac{\partial a}{\partial \bar z}\,\extd \bar z=\extd \bar z\, \Y^m+\bar z\,
 \frac{\partial (\Y^m)}{\partial \bar z}\,\extd\bar z$, so
$\frac{\partial a}{\partial \bar z}=\Y^m+\bar z\,  \frac{\partial (\Y^m)}{\partial \bar z}$. Now from (\ref{missedp}),
\begin{align*}
\tfrac{\partial a}{\partial \bar z}=&\ \Y^m-q^2\,\bar zz\,[m]_{q^2}\,\Y^m \cr
=&\ \Y^m+q^2\,\Y\,[m]_{q^2}\,\Y^{m-1}-q^2\,[m]_{q^2}\,\Y^{m-1} \cr
=&\ [m+1]_{q^2}\Y^m-q^2\,[m]_{q^2}\Y^{m-1}\ ,\cr
q^{-2m} \tfrac{\partial a}{\partial \bar z}=&\ [m+1]_{q^{-2}}\Y^m-[m]_{q^{-2}}\Y^{m-1}\ .
\end{align*}
Now the formula for the integral shows that $\dirac$ is hermitian  (The use of the hermitian property for the inner product means that we have checked this for all cases.). However note that if we were to set $a=\bar z$, then the condition would not be satisfied. There is a condition on the domain of $\dirac$ which would classically include functions vanishing on the boundary of the disk -- we shall not pursue this matter further here. 

 The condition that $\mathcal{J}$ is an isometry would require equality under the integral of
 \[ \<\overline{\J(a s)}, \J(bs)\>=\mu \delta^2 a \Y b^*=q^{-1}a\Y b^*,\quad \<\overline{bs},as\>=q^{-|a|-|b|}b^* \Y a\]
on homogeneous elements. The second expression integrates to zero unless $|a|=|b|$ so gives $\varsigma(b^*)\Y a$. From this and the above twisted trace
property of the integral, we can conclude that 
\[  \<\!\<\overline{\J (as)},\J(bs)\>\!\>= q^{-1}\<\!\< \overline{\varsigma^{-1} (b)s},as\>\!\>,\quad \forall a,b\in \C_q[D]\]
much as in the spirit of the $q$-sphere example. The same applies with $s$ replaced by $\bar s$ and $q^{-1}$ by $q$.  Finally, we have some freedom
to rescale the generators $s,\bar s$ and using this we can without loss of generality assume $\alpha=1$ and $\beta<0$ say when $q>0$. In that
case we have a 1-parameter family of Dirac operators by our construction with $\delta=\sqrt{-q^{-1}\beta}, \mu=-\beta^{-1}$.

\section{Holomorphic bimodules and Chern connections}
An integrable almost complex structure (see \cite{BegSmiComplex}) on a star algebra $A$  with 
star differential calculus $(\Omega,\extd,\wedge)$ is a bimodule map $J:\Omega^1\to \Omega^1$ with $J^2=-\id$ obeying certain conditions. This basically is a decomposition of bimodules $\Omega^n=\oplus_{p+q=n}
\Omega^{p,q}$ with $\extd=\pd+\pdol$ where $\pd:\Omega^{p,q}\to \Omega^{p+1,q}$ and 
$\pdol:\Omega^{p,q}\to \Omega^{p,q+1}$ with $\pd^2=\pdol{}^2=0$ and $\Omega^{p,q}\wedge\Omega^{p',q'}\subset \Omega^{p+p',q+q'}$ . The direct sum gives projection maps 
$\pi^{p,q}:\Omega^{p+q}\to\Omega^{p,q}$.

We can then define the notion of a holomorphic bundle $E$ but note that in \cite{BegSmiComplex} the basic definition of this is given in terms of a left holomorphic section. This awkwardly lends itself to looking at hermitian inner products with the antilinear side on the right, which is opposite to the usual convention for the Dirac operator. However, in the spinor bundle case we can use the antilinear isomorphism $j:
\CS\to\overline{\CS}$ to swap the side of the antilinear part by commutativity of the following diagram:
\begin{align} \label{sideswitch}
\xymatrix{
\overline{\CS}\otimes_A \CS \ar[r]^{\<,\>}    & A \\
\CS \otimes_A \overline{\CS} \ar[ur]_{\<,\>} \ar[u]^{j\tens j^{-1}} 
  }
\end{align}
In general there is no necessity for a holomorphic bundle to have any obvious antilinear isomorphism. 
 
\subsection{Holomorphic bimodules}

  A left $\pdol$-connection $\pdol_E:E\to \Omega^{0,1}A\tens_A E$ on a left module $E$ is a linear map satisfying the left $\pdol$-Liebniz rule, for $e\in E$ and $a\in A$
\[
\pdol_E(a.e)=\pdol a\tens e+a.\pdol_E(e)\ .
\]

\begin{definition}\cite{BegSmiComplex} 
A holomorphic structure on a left $A$-module
 $E$ is is a left $\pdol$-connection $\pdol_E:E\to \Omega^{0,1}A\tens_A E$ with vanishing holomorphic curvature, i.e.\
 $(\pdol\tens\id-\id\wedge\pdol_E)\pdol_E:E\to \Omega^{0,2}A\tens_A E$ vanishes. Then we call $(E,\pdol_E)$ a left holomorphic module.
 If in addition there is a bimodule map 
$\sigma_E:E\tens_A\Omega^{0,1}A\to \Omega^{0,1}A\tens_A E$ so that
$\pdol_E(e.a)=\pdol_E(e).a+\sigma_E(e\tens\pdol a)$ for all $a\in A$ and $e\in E$, we say that $(E,\pdol_E,\sigma_E)$ is a left holomorphic bimodule. 
\end{definition}

Now suppose that we have a hermitian inner product on a holomorphic left module. Classically the $\pdol$-connection given by the holomorphic structure can be extended to a unique connection $\nabla_E:E\to\Omega^1\tens_A E$ preserving the hermitian inner product and with curvature only mapping to $\Omega^{1,1}\tens_A E$, the Chern connection. We shall show that this is also the case in noncommutative geometry. 
We give two constructions of the same connection, as both has its advantages.
For one we require some material on duals and coevaluations, and for the other we consider Christoffel symbols.

We begin with a hermitian inner product on $E$, $\<,\>:E\tens_A\overline{E}\to A$, and 
 we say that the left connection $\nabla_E:E\to \Omega^{1}A\tens_A E$ preserves the metric if
\begin{eqnarray} \label{hermcon1}
\extd\,\<,\>=(\id\tens\<,\>)(\nabla_E\tens\id)+(\<,\>\tens\id)(\id\tens \tilde\nabla):E\tens  \overline{E}\to \Omega^1A\ ,
\end{eqnarray}
where $\tilde\nabla:\overline{E}\to \overline{E}\tens_A\Omega^1A$ is the right connection constructed from $\nabla_E$ by
$\tilde\nabla(\overline{e})=\overline{f}\tens\kappa^*$, where $\nabla_E(e)=\kappa\tens f$.

\subsection{Duals and coevaluations}
For an $A$-bimodule $E$ we use the notation $E^\circ={}_A\Hom(E,A)$, and there is an evaluation bimodule map
$\ev_E:E\tens_A E^\circ\to A$. If $E$ is finitely generated projective (fgp for short) as a left module,  then there is a coevaluation map
$\coev_E:A\to E^\circ\tens_A E$ (written items of a `dual basis') so that
\begin{eqnarray*}
(\id\tens\ev_E)(\coev_E\tens \id)=\id:E^\circ\to E^\circ\ ,\quad (\ev_E\tens\id)(\id\tens \coev_E)=\id:E\to E\ .
\end{eqnarray*}

Choosing a side, we suppose that $\<,\>:E\tens_A\overline{E}\to A$ is a hermitian inner product on $E$ (and therefore that it is a bimodule map). This hermitian inner product  is called non-degenerate if there is a bimodule isomorphism $G:\overline{E}\to E^\circ$ so that
$\<,\>=\ev\circ(\id\tens G)$. In this case we can define an `inverse inner product' $\<,\>^{-1}:A\to \overline{E}\tens_A E$ by $\<,\>^{-1}=(G^{-1}\tens\id)\coev_E$.

\begin{proposition} \label{unichernnog}
Suppose that  $(E,\pdol_E)$ is a holomorphic left module, where $E$ is finitely generated projective as a left $A$-module
and  $\<,\>:E\tens\overline{E}\to A$ is a non-degenerate hermitian inner product on $E$. Then there is a unique left connection
$\nabla_E:E\to\Omega^1 \tens_A E$,  preserving the hermitian metric and obeying
$(\pi^{0,1}\tens\id)\nabla_E=\pdol_E$.

The $(\pi^{1,0}\tens\id)\nabla_E=\pd_E$ part is given by the following diagram, where $\tilde\pd: \overline{E}\to
\overline{E}\tens_A \Omega^{1,0}$ is the right $\pd$-covariant derivative defined by $\tilde\pd(\overline{e})=\overline{f}\tens \kappa^*$, where $\kappa\tens f=\pdol_E(e)$: 

\unitlength 0.5 mm
\begin{picture}(180,60)(-10,27)
\linethickness{0.3mm}
\put(20,60){\line(0,1){25}}
\linethickness{0.3mm}
\multiput(29.99,50.5)(0.01,-0.5){1}{\line(0,-1){0.5}}
\multiput(29.95,51)(0.04,-0.5){1}{\line(0,-1){0.5}}
\multiput(29.89,51.49)(0.06,-0.49){1}{\line(0,-1){0.49}}
\multiput(29.8,51.98)(0.09,-0.49){1}{\line(0,-1){0.49}}
\multiput(29.69,52.47)(0.11,-0.49){1}{\line(0,-1){0.49}}
\multiput(29.56,52.95)(0.14,-0.48){1}{\line(0,-1){0.48}}
\multiput(29.4,53.42)(0.16,-0.47){1}{\line(0,-1){0.47}}
\multiput(29.21,53.88)(0.09,-0.23){2}{\line(0,-1){0.23}}
\multiput(29.01,54.34)(0.1,-0.23){2}{\line(0,-1){0.23}}
\multiput(28.78,54.78)(0.11,-0.22){2}{\line(0,-1){0.22}}
\multiput(28.53,55.21)(0.12,-0.22){2}{\line(0,-1){0.22}}
\multiput(28.26,55.63)(0.14,-0.21){2}{\line(0,-1){0.21}}
\multiput(27.97,56.04)(0.15,-0.2){2}{\line(0,-1){0.2}}
\multiput(27.66,56.43)(0.1,-0.13){3}{\line(0,-1){0.13}}
\multiput(27.33,56.8)(0.11,-0.12){3}{\line(0,-1){0.12}}
\multiput(26.98,57.16)(0.12,-0.12){3}{\line(0,-1){0.12}}
\multiput(26.62,57.5)(0.12,-0.11){3}{\line(1,0){0.12}}
\multiput(26.23,57.82)(0.13,-0.11){3}{\line(1,0){0.13}}
\multiput(25.84,58.12)(0.13,-0.1){3}{\line(1,0){0.13}}
\multiput(25.43,58.4)(0.21,-0.14){2}{\line(1,0){0.21}}
\multiput(25,58.66)(0.21,-0.13){2}{\line(1,0){0.21}}
\multiput(24.56,58.9)(0.22,-0.12){2}{\line(1,0){0.22}}
\multiput(24.11,59.12)(0.22,-0.11){2}{\line(1,0){0.22}}
\multiput(23.65,59.31)(0.23,-0.1){2}{\line(1,0){0.23}}
\multiput(23.18,59.48)(0.47,-0.17){1}{\line(1,0){0.47}}
\multiput(22.71,59.63)(0.48,-0.15){1}{\line(1,0){0.48}}
\multiput(22.23,59.75)(0.48,-0.12){1}{\line(1,0){0.48}}
\multiput(21.74,59.85)(0.49,-0.1){1}{\line(1,0){0.49}}
\multiput(21.24,59.92)(0.49,-0.07){1}{\line(1,0){0.49}}
\multiput(20.75,59.97)(0.5,-0.05){1}{\line(1,0){0.5}}
\multiput(20.25,60)(0.5,-0.02){1}{\line(1,0){0.5}}
\put(19.75,60){\line(1,0){0.5}}
\multiput(19.25,59.97)(0.5,0.02){1}{\line(1,0){0.5}}
\multiput(18.76,59.92)(0.5,0.05){1}{\line(1,0){0.5}}
\multiput(18.26,59.85)(0.49,0.07){1}{\line(1,0){0.49}}
\multiput(17.77,59.75)(0.49,0.1){1}{\line(1,0){0.49}}
\multiput(17.29,59.63)(0.48,0.12){1}{\line(1,0){0.48}}
\multiput(16.82,59.48)(0.48,0.15){1}{\line(1,0){0.48}}
\multiput(16.35,59.31)(0.47,0.17){1}{\line(1,0){0.47}}
\multiput(15.89,59.12)(0.23,0.1){2}{\line(1,0){0.23}}
\multiput(15.44,58.9)(0.22,0.11){2}{\line(1,0){0.22}}
\multiput(15,58.66)(0.22,0.12){2}{\line(1,0){0.22}}
\multiput(14.57,58.4)(0.21,0.13){2}{\line(1,0){0.21}}
\multiput(14.16,58.12)(0.21,0.14){2}{\line(1,0){0.21}}
\multiput(13.77,57.82)(0.13,0.1){3}{\line(1,0){0.13}}
\multiput(13.38,57.5)(0.13,0.11){3}{\line(1,0){0.13}}
\multiput(13.02,57.16)(0.12,0.11){3}{\line(1,0){0.12}}
\multiput(12.67,56.8)(0.12,0.12){3}{\line(0,1){0.12}}
\multiput(12.34,56.43)(0.11,0.12){3}{\line(0,1){0.12}}
\multiput(12.03,56.04)(0.1,0.13){3}{\line(0,1){0.13}}
\multiput(11.74,55.63)(0.15,0.2){2}{\line(0,1){0.2}}
\multiput(11.47,55.21)(0.14,0.21){2}{\line(0,1){0.21}}
\multiput(11.22,54.78)(0.12,0.22){2}{\line(0,1){0.22}}
\multiput(10.99,54.34)(0.11,0.22){2}{\line(0,1){0.22}}
\multiput(10.79,53.88)(0.1,0.23){2}{\line(0,1){0.23}}
\multiput(10.6,53.42)(0.09,0.23){2}{\line(0,1){0.23}}
\multiput(10.44,52.95)(0.16,0.47){1}{\line(0,1){0.47}}
\multiput(10.31,52.47)(0.14,0.48){1}{\line(0,1){0.48}}
\multiput(10.2,51.98)(0.11,0.49){1}{\line(0,1){0.49}}
\multiput(10.11,51.49)(0.09,0.49){1}{\line(0,1){0.49}}
\multiput(10.05,51)(0.06,0.49){1}{\line(0,1){0.49}}
\multiput(10.01,50.5)(0.04,0.5){1}{\line(0,1){0.5}}
\multiput(10,50)(0.01,0.5){1}{\line(0,1){0.5}}

\linethickness{0.3mm}
\put(60,60){\line(0,1){25}}
\linethickness{0.3mm}
\multiput(60,60)(0.01,-0.5){1}{\line(0,-1){0.5}}
\multiput(60.01,59.5)(0.04,-0.5){1}{\line(0,-1){0.5}}
\multiput(60.05,59)(0.06,-0.49){1}{\line(0,-1){0.49}}
\multiput(60.11,58.51)(0.09,-0.49){1}{\line(0,-1){0.49}}
\multiput(60.2,58.02)(0.11,-0.49){1}{\line(0,-1){0.49}}
\multiput(60.31,57.53)(0.14,-0.48){1}{\line(0,-1){0.48}}
\multiput(60.44,57.05)(0.16,-0.47){1}{\line(0,-1){0.47}}
\multiput(60.6,56.58)(0.09,-0.23){2}{\line(0,-1){0.23}}
\multiput(60.79,56.12)(0.1,-0.23){2}{\line(0,-1){0.23}}
\multiput(60.99,55.66)(0.11,-0.22){2}{\line(0,-1){0.22}}
\multiput(61.22,55.22)(0.12,-0.22){2}{\line(0,-1){0.22}}
\multiput(61.47,54.79)(0.14,-0.21){2}{\line(0,-1){0.21}}
\multiput(61.74,54.37)(0.15,-0.2){2}{\line(0,-1){0.2}}
\multiput(62.03,53.96)(0.1,-0.13){3}{\line(0,-1){0.13}}
\multiput(62.34,53.57)(0.11,-0.12){3}{\line(0,-1){0.12}}
\multiput(62.67,53.2)(0.12,-0.12){3}{\line(0,-1){0.12}}
\multiput(63.02,52.84)(0.12,-0.11){3}{\line(1,0){0.12}}
\multiput(63.38,52.5)(0.13,-0.11){3}{\line(1,0){0.13}}
\multiput(63.77,52.18)(0.13,-0.1){3}{\line(1,0){0.13}}
\multiput(64.16,51.88)(0.21,-0.14){2}{\line(1,0){0.21}}
\multiput(64.57,51.6)(0.21,-0.13){2}{\line(1,0){0.21}}
\multiput(65,51.34)(0.22,-0.12){2}{\line(1,0){0.22}}
\multiput(65.44,51.1)(0.22,-0.11){2}{\line(1,0){0.22}}
\multiput(65.89,50.88)(0.23,-0.1){2}{\line(1,0){0.23}}
\multiput(66.35,50.69)(0.47,-0.17){1}{\line(1,0){0.47}}
\multiput(66.82,50.52)(0.48,-0.15){1}{\line(1,0){0.48}}
\multiput(67.29,50.37)(0.48,-0.12){1}{\line(1,0){0.48}}
\multiput(67.77,50.25)(0.49,-0.1){1}{\line(1,0){0.49}}
\multiput(68.26,50.15)(0.49,-0.07){1}{\line(1,0){0.49}}
\multiput(68.76,50.08)(0.5,-0.05){1}{\line(1,0){0.5}}
\multiput(69.25,50.03)(0.5,-0.02){1}{\line(1,0){0.5}}
\put(69.75,50){\line(1,0){0.5}}
\multiput(70.25,50)(0.5,0.02){1}{\line(1,0){0.5}}
\multiput(70.75,50.03)(0.5,0.05){1}{\line(1,0){0.5}}
\multiput(71.24,50.08)(0.49,0.07){1}{\line(1,0){0.49}}
\multiput(71.74,50.15)(0.49,0.1){1}{\line(1,0){0.49}}
\multiput(72.23,50.25)(0.48,0.12){1}{\line(1,0){0.48}}
\multiput(72.71,50.37)(0.48,0.15){1}{\line(1,0){0.48}}
\multiput(73.18,50.52)(0.47,0.17){1}{\line(1,0){0.47}}
\multiput(73.65,50.69)(0.23,0.1){2}{\line(1,0){0.23}}
\multiput(74.11,50.88)(0.22,0.11){2}{\line(1,0){0.22}}
\multiput(74.56,51.1)(0.22,0.12){2}{\line(1,0){0.22}}
\multiput(75,51.34)(0.21,0.13){2}{\line(1,0){0.21}}
\multiput(75.43,51.6)(0.21,0.14){2}{\line(1,0){0.21}}
\multiput(75.84,51.88)(0.13,0.1){3}{\line(1,0){0.13}}
\multiput(76.23,52.18)(0.13,0.11){3}{\line(1,0){0.13}}
\multiput(76.62,52.5)(0.12,0.11){3}{\line(1,0){0.12}}
\multiput(76.98,52.84)(0.12,0.12){3}{\line(0,1){0.12}}
\multiput(77.33,53.2)(0.11,0.12){3}{\line(0,1){0.12}}
\multiput(77.66,53.57)(0.1,0.13){3}{\line(0,1){0.13}}
\multiput(77.97,53.96)(0.15,0.2){2}{\line(0,1){0.2}}
\multiput(78.26,54.37)(0.14,0.21){2}{\line(0,1){0.21}}
\multiput(78.53,54.79)(0.12,0.22){2}{\line(0,1){0.22}}
\multiput(78.78,55.22)(0.11,0.22){2}{\line(0,1){0.22}}
\multiput(79.01,55.66)(0.1,0.23){2}{\line(0,1){0.23}}
\multiput(79.21,56.12)(0.09,0.23){2}{\line(0,1){0.23}}
\multiput(79.4,56.58)(0.16,0.47){1}{\line(0,1){0.47}}
\multiput(79.56,57.05)(0.14,0.48){1}{\line(0,1){0.48}}
\multiput(79.69,57.53)(0.11,0.49){1}{\line(0,1){0.49}}
\multiput(79.8,58.02)(0.09,0.49){1}{\line(0,1){0.49}}
\multiput(79.89,58.51)(0.06,0.49){1}{\line(0,1){0.49}}
\multiput(79.95,59)(0.04,0.5){1}{\line(0,1){0.5}}
\multiput(79.99,59.5)(0.01,0.5){1}{\line(0,1){0.5}}

\linethickness{0.3mm}
\multiput(99.99,60.5)(0.01,-0.5){1}{\line(0,-1){0.5}}
\multiput(99.95,61)(0.04,-0.5){1}{\line(0,-1){0.5}}
\multiput(99.89,61.49)(0.06,-0.49){1}{\line(0,-1){0.49}}
\multiput(99.8,61.98)(0.09,-0.49){1}{\line(0,-1){0.49}}
\multiput(99.69,62.47)(0.11,-0.49){1}{\line(0,-1){0.49}}
\multiput(99.56,62.95)(0.14,-0.48){1}{\line(0,-1){0.48}}
\multiput(99.4,63.42)(0.16,-0.47){1}{\line(0,-1){0.47}}
\multiput(99.21,63.88)(0.09,-0.23){2}{\line(0,-1){0.23}}
\multiput(99.01,64.34)(0.1,-0.23){2}{\line(0,-1){0.23}}
\multiput(98.78,64.78)(0.11,-0.22){2}{\line(0,-1){0.22}}
\multiput(98.53,65.21)(0.12,-0.22){2}{\line(0,-1){0.22}}
\multiput(98.26,65.63)(0.14,-0.21){2}{\line(0,-1){0.21}}
\multiput(97.97,66.04)(0.15,-0.2){2}{\line(0,-1){0.2}}
\multiput(97.66,66.43)(0.1,-0.13){3}{\line(0,-1){0.13}}
\multiput(97.33,66.8)(0.11,-0.12){3}{\line(0,-1){0.12}}
\multiput(96.98,67.16)(0.12,-0.12){3}{\line(0,-1){0.12}}
\multiput(96.62,67.5)(0.12,-0.11){3}{\line(1,0){0.12}}
\multiput(96.23,67.82)(0.13,-0.11){3}{\line(1,0){0.13}}
\multiput(95.84,68.12)(0.13,-0.1){3}{\line(1,0){0.13}}
\multiput(95.43,68.4)(0.21,-0.14){2}{\line(1,0){0.21}}
\multiput(95,68.66)(0.21,-0.13){2}{\line(1,0){0.21}}
\multiput(94.56,68.9)(0.22,-0.12){2}{\line(1,0){0.22}}
\multiput(94.11,69.12)(0.22,-0.11){2}{\line(1,0){0.22}}
\multiput(93.65,69.31)(0.23,-0.1){2}{\line(1,0){0.23}}
\multiput(93.18,69.48)(0.47,-0.17){1}{\line(1,0){0.47}}
\multiput(92.71,69.63)(0.48,-0.15){1}{\line(1,0){0.48}}
\multiput(92.23,69.75)(0.48,-0.12){1}{\line(1,0){0.48}}
\multiput(91.74,69.85)(0.49,-0.1){1}{\line(1,0){0.49}}
\multiput(91.24,69.92)(0.49,-0.07){1}{\line(1,0){0.49}}
\multiput(90.75,69.97)(0.5,-0.05){1}{\line(1,0){0.5}}
\multiput(90.25,70)(0.5,-0.02){1}{\line(1,0){0.5}}
\put(89.75,70){\line(1,0){0.5}}
\multiput(89.25,69.97)(0.5,0.02){1}{\line(1,0){0.5}}
\multiput(88.76,69.92)(0.5,0.05){1}{\line(1,0){0.5}}
\multiput(88.26,69.85)(0.49,0.07){1}{\line(1,0){0.49}}
\multiput(87.77,69.75)(0.49,0.1){1}{\line(1,0){0.49}}
\multiput(87.29,69.63)(0.48,0.12){1}{\line(1,0){0.48}}
\multiput(86.82,69.48)(0.48,0.15){1}{\line(1,0){0.48}}
\multiput(86.35,69.31)(0.47,0.17){1}{\line(1,0){0.47}}
\multiput(85.89,69.12)(0.23,0.1){2}{\line(1,0){0.23}}
\multiput(85.44,68.9)(0.22,0.11){2}{\line(1,0){0.22}}
\multiput(85,68.66)(0.22,0.12){2}{\line(1,0){0.22}}
\multiput(84.57,68.4)(0.21,0.13){2}{\line(1,0){0.21}}
\multiput(84.16,68.12)(0.21,0.14){2}{\line(1,0){0.21}}
\multiput(83.77,67.82)(0.13,0.1){3}{\line(1,0){0.13}}
\multiput(83.38,67.5)(0.13,0.11){3}{\line(1,0){0.13}}
\multiput(83.02,67.16)(0.12,0.11){3}{\line(1,0){0.12}}
\multiput(82.67,66.8)(0.12,0.12){3}{\line(0,1){0.12}}
\multiput(82.34,66.43)(0.11,0.12){3}{\line(0,1){0.12}}
\multiput(82.03,66.04)(0.1,0.13){3}{\line(0,1){0.13}}
\multiput(81.74,65.63)(0.15,0.2){2}{\line(0,1){0.2}}
\multiput(81.47,65.21)(0.14,0.21){2}{\line(0,1){0.21}}
\multiput(81.22,64.78)(0.12,0.22){2}{\line(0,1){0.22}}
\multiput(80.99,64.34)(0.11,0.22){2}{\line(0,1){0.22}}
\multiput(80.79,63.88)(0.1,0.23){2}{\line(0,1){0.23}}
\multiput(80.6,63.42)(0.09,0.23){2}{\line(0,1){0.23}}
\multiput(80.44,62.95)(0.16,0.47){1}{\line(0,1){0.47}}
\multiput(80.31,62.47)(0.14,0.48){1}{\line(0,1){0.48}}
\multiput(80.2,61.98)(0.11,0.49){1}{\line(0,1){0.49}}
\multiput(80.11,61.49)(0.09,0.49){1}{\line(0,1){0.49}}
\multiput(80.05,61)(0.06,0.49){1}{\line(0,1){0.49}}
\multiput(80.01,60.5)(0.04,0.5){1}{\line(0,1){0.5}}
\multiput(80,60)(0.01,0.5){1}{\line(0,1){0.5}}

\linethickness{0.3mm}
\put(70,45){\line(0,1){5}}
\linethickness{0.3mm}
\put(70,25){\line(0,1){10}}
\linethickness{0.3mm}
\put(100,25){\line(0,1){35}}
\linethickness{0.3mm}
\put(130,50){\line(0,1){35}}
\linethickness{0.3mm}
\multiput(130,50)(0.01,-0.5){1}{\line(0,-1){0.5}}
\multiput(130.01,49.5)(0.04,-0.5){1}{\line(0,-1){0.5}}
\multiput(130.05,49)(0.06,-0.49){1}{\line(0,-1){0.49}}
\multiput(130.11,48.51)(0.09,-0.49){1}{\line(0,-1){0.49}}
\multiput(130.2,48.02)(0.11,-0.49){1}{\line(0,-1){0.49}}
\multiput(130.31,47.53)(0.14,-0.48){1}{\line(0,-1){0.48}}
\multiput(130.44,47.05)(0.16,-0.47){1}{\line(0,-1){0.47}}
\multiput(130.6,46.58)(0.09,-0.23){2}{\line(0,-1){0.23}}
\multiput(130.79,46.12)(0.1,-0.23){2}{\line(0,-1){0.23}}
\multiput(130.99,45.66)(0.11,-0.22){2}{\line(0,-1){0.22}}
\multiput(131.22,45.22)(0.12,-0.22){2}{\line(0,-1){0.22}}
\multiput(131.47,44.79)(0.14,-0.21){2}{\line(0,-1){0.21}}
\multiput(131.74,44.37)(0.15,-0.2){2}{\line(0,-1){0.2}}
\multiput(132.03,43.96)(0.1,-0.13){3}{\line(0,-1){0.13}}
\multiput(132.34,43.57)(0.11,-0.12){3}{\line(0,-1){0.12}}
\multiput(132.67,43.2)(0.12,-0.12){3}{\line(0,-1){0.12}}
\multiput(133.02,42.84)(0.12,-0.11){3}{\line(1,0){0.12}}
\multiput(133.38,42.5)(0.13,-0.11){3}{\line(1,0){0.13}}
\multiput(133.77,42.18)(0.13,-0.1){3}{\line(1,0){0.13}}
\multiput(134.16,41.88)(0.21,-0.14){2}{\line(1,0){0.21}}
\multiput(134.57,41.6)(0.21,-0.13){2}{\line(1,0){0.21}}
\multiput(135,41.34)(0.22,-0.12){2}{\line(1,0){0.22}}
\multiput(135.44,41.1)(0.22,-0.11){2}{\line(1,0){0.22}}
\multiput(135.89,40.88)(0.23,-0.1){2}{\line(1,0){0.23}}
\multiput(136.35,40.69)(0.47,-0.17){1}{\line(1,0){0.47}}
\multiput(136.82,40.52)(0.48,-0.15){1}{\line(1,0){0.48}}
\multiput(137.29,40.37)(0.48,-0.12){1}{\line(1,0){0.48}}
\multiput(137.77,40.25)(0.49,-0.1){1}{\line(1,0){0.49}}
\multiput(138.26,40.15)(0.49,-0.07){1}{\line(1,0){0.49}}
\multiput(138.76,40.08)(0.5,-0.05){1}{\line(1,0){0.5}}
\multiput(139.25,40.03)(0.5,-0.02){1}{\line(1,0){0.5}}
\put(139.75,40){\line(1,0){0.5}}
\multiput(140.25,40)(0.5,0.02){1}{\line(1,0){0.5}}
\multiput(140.75,40.03)(0.5,0.05){1}{\line(1,0){0.5}}
\multiput(141.24,40.08)(0.49,0.07){1}{\line(1,0){0.49}}
\multiput(141.74,40.15)(0.49,0.1){1}{\line(1,0){0.49}}
\multiput(142.23,40.25)(0.48,0.12){1}{\line(1,0){0.48}}
\multiput(142.71,40.37)(0.48,0.15){1}{\line(1,0){0.48}}
\multiput(143.18,40.52)(0.47,0.17){1}{\line(1,0){0.47}}
\multiput(143.65,40.69)(0.23,0.1){2}{\line(1,0){0.23}}
\multiput(144.11,40.88)(0.22,0.11){2}{\line(1,0){0.22}}
\multiput(144.56,41.1)(0.22,0.12){2}{\line(1,0){0.22}}
\multiput(145,41.34)(0.21,0.13){2}{\line(1,0){0.21}}
\multiput(145.43,41.6)(0.21,0.14){2}{\line(1,0){0.21}}
\multiput(145.84,41.88)(0.13,0.1){3}{\line(1,0){0.13}}
\multiput(146.23,42.18)(0.13,0.11){3}{\line(1,0){0.13}}
\multiput(146.62,42.5)(0.12,0.11){3}{\line(1,0){0.12}}
\multiput(146.98,42.84)(0.12,0.12){3}{\line(0,1){0.12}}
\multiput(147.33,43.2)(0.11,0.12){3}{\line(0,1){0.12}}
\multiput(147.66,43.57)(0.1,0.13){3}{\line(0,1){0.13}}
\multiput(147.97,43.96)(0.15,0.2){2}{\line(0,1){0.2}}
\multiput(148.26,44.37)(0.14,0.21){2}{\line(0,1){0.21}}
\multiput(148.53,44.79)(0.12,0.22){2}{\line(0,1){0.22}}
\multiput(148.78,45.22)(0.11,0.22){2}{\line(0,1){0.22}}
\multiput(149.01,45.66)(0.1,0.23){2}{\line(0,1){0.23}}
\multiput(149.21,46.12)(0.09,0.23){2}{\line(0,1){0.23}}
\multiput(149.4,46.58)(0.16,0.47){1}{\line(0,1){0.47}}
\multiput(149.56,47.05)(0.14,0.48){1}{\line(0,1){0.48}}
\multiput(149.69,47.53)(0.11,0.49){1}{\line(0,1){0.49}}
\multiput(149.8,48.02)(0.09,0.49){1}{\line(0,1){0.49}}
\multiput(149.89,48.51)(0.06,0.49){1}{\line(0,1){0.49}}
\multiput(149.95,49)(0.04,0.5){1}{\line(0,1){0.5}}
\multiput(149.99,49.5)(0.01,0.5){1}{\line(0,1){0.5}}

\linethickness{0.3mm}
\multiput(169.99,50.5)(0.01,-0.5){1}{\line(0,-1){0.5}}
\multiput(169.95,51)(0.04,-0.5){1}{\line(0,-1){0.5}}
\multiput(169.89,51.49)(0.06,-0.49){1}{\line(0,-1){0.49}}
\multiput(169.8,51.98)(0.09,-0.49){1}{\line(0,-1){0.49}}
\multiput(169.69,52.47)(0.11,-0.49){1}{\line(0,-1){0.49}}
\multiput(169.56,52.95)(0.14,-0.48){1}{\line(0,-1){0.48}}
\multiput(169.4,53.42)(0.16,-0.47){1}{\line(0,-1){0.47}}
\multiput(169.21,53.88)(0.09,-0.23){2}{\line(0,-1){0.23}}
\multiput(169.01,54.34)(0.1,-0.23){2}{\line(0,-1){0.23}}
\multiput(168.78,54.78)(0.11,-0.22){2}{\line(0,-1){0.22}}
\multiput(168.53,55.21)(0.12,-0.22){2}{\line(0,-1){0.22}}
\multiput(168.26,55.63)(0.14,-0.21){2}{\line(0,-1){0.21}}
\multiput(167.97,56.04)(0.15,-0.2){2}{\line(0,-1){0.2}}
\multiput(167.66,56.43)(0.1,-0.13){3}{\line(0,-1){0.13}}
\multiput(167.33,56.8)(0.11,-0.12){3}{\line(0,-1){0.12}}
\multiput(166.98,57.16)(0.12,-0.12){3}{\line(0,-1){0.12}}
\multiput(166.62,57.5)(0.12,-0.11){3}{\line(1,0){0.12}}
\multiput(166.23,57.82)(0.13,-0.11){3}{\line(1,0){0.13}}
\multiput(165.84,58.12)(0.13,-0.1){3}{\line(1,0){0.13}}
\multiput(165.43,58.4)(0.21,-0.14){2}{\line(1,0){0.21}}
\multiput(165,58.66)(0.21,-0.13){2}{\line(1,0){0.21}}
\multiput(164.56,58.9)(0.22,-0.12){2}{\line(1,0){0.22}}
\multiput(164.11,59.12)(0.22,-0.11){2}{\line(1,0){0.22}}
\multiput(163.65,59.31)(0.23,-0.1){2}{\line(1,0){0.23}}
\multiput(163.18,59.48)(0.47,-0.17){1}{\line(1,0){0.47}}
\multiput(162.71,59.63)(0.48,-0.15){1}{\line(1,0){0.48}}
\multiput(162.23,59.75)(0.48,-0.12){1}{\line(1,0){0.48}}
\multiput(161.74,59.85)(0.49,-0.1){1}{\line(1,0){0.49}}
\multiput(161.24,59.92)(0.49,-0.07){1}{\line(1,0){0.49}}
\multiput(160.75,59.97)(0.5,-0.05){1}{\line(1,0){0.5}}
\multiput(160.25,60)(0.5,-0.02){1}{\line(1,0){0.5}}
\put(159.75,60){\line(1,0){0.5}}
\multiput(159.25,59.97)(0.5,0.02){1}{\line(1,0){0.5}}
\multiput(158.76,59.92)(0.5,0.05){1}{\line(1,0){0.5}}
\multiput(158.26,59.85)(0.49,0.07){1}{\line(1,0){0.49}}
\multiput(157.77,59.75)(0.49,0.1){1}{\line(1,0){0.49}}
\multiput(157.29,59.63)(0.48,0.12){1}{\line(1,0){0.48}}
\multiput(156.82,59.48)(0.48,0.15){1}{\line(1,0){0.48}}
\multiput(156.35,59.31)(0.47,0.17){1}{\line(1,0){0.47}}
\multiput(155.89,59.12)(0.23,0.1){2}{\line(1,0){0.23}}
\multiput(155.44,58.9)(0.22,0.11){2}{\line(1,0){0.22}}
\multiput(155,58.66)(0.22,0.12){2}{\line(1,0){0.22}}
\multiput(154.57,58.4)(0.21,0.13){2}{\line(1,0){0.21}}
\multiput(154.16,58.12)(0.21,0.14){2}{\line(1,0){0.21}}
\multiput(153.77,57.82)(0.13,0.1){3}{\line(1,0){0.13}}
\multiput(153.38,57.5)(0.13,0.11){3}{\line(1,0){0.13}}
\multiput(153.02,57.16)(0.12,0.11){3}{\line(1,0){0.12}}
\multiput(152.67,56.8)(0.12,0.12){3}{\line(0,1){0.12}}
\multiput(152.34,56.43)(0.11,0.12){3}{\line(0,1){0.12}}
\multiput(152.03,56.04)(0.1,0.13){3}{\line(0,1){0.13}}
\multiput(151.74,55.63)(0.15,0.2){2}{\line(0,1){0.2}}
\multiput(151.47,55.21)(0.14,0.21){2}{\line(0,1){0.21}}
\multiput(151.22,54.78)(0.12,0.22){2}{\line(0,1){0.22}}
\multiput(150.99,54.34)(0.11,0.22){2}{\line(0,1){0.22}}
\multiput(150.79,53.88)(0.1,0.23){2}{\line(0,1){0.23}}
\multiput(150.6,53.42)(0.09,0.23){2}{\line(0,1){0.23}}
\multiput(150.44,52.95)(0.16,0.47){1}{\line(0,1){0.47}}
\multiput(150.31,52.47)(0.14,0.48){1}{\line(0,1){0.48}}
\multiput(150.2,51.98)(0.11,0.49){1}{\line(0,1){0.49}}
\multiput(150.11,51.49)(0.09,0.49){1}{\line(0,1){0.49}}
\multiput(150.05,51)(0.06,0.49){1}{\line(0,1){0.49}}
\multiput(150.01,50.5)(0.04,0.5){1}{\line(0,1){0.5}}
\multiput(150,50)(0.01,0.5){1}{\line(0,1){0.5}}

\linethickness{0.3mm}
\multiput(179.99,70.5)(0.01,-0.5){1}{\line(0,-1){0.5}}
\multiput(179.95,71)(0.04,-0.5){1}{\line(0,-1){0.5}}
\multiput(179.89,71.49)(0.06,-0.49){1}{\line(0,-1){0.49}}
\multiput(179.8,71.98)(0.09,-0.49){1}{\line(0,-1){0.49}}
\multiput(179.69,72.47)(0.11,-0.49){1}{\line(0,-1){0.49}}
\multiput(179.56,72.95)(0.14,-0.48){1}{\line(0,-1){0.48}}
\multiput(179.4,73.42)(0.16,-0.47){1}{\line(0,-1){0.47}}
\multiput(179.21,73.88)(0.09,-0.23){2}{\line(0,-1){0.23}}
\multiput(179.01,74.34)(0.1,-0.23){2}{\line(0,-1){0.23}}
\multiput(178.78,74.78)(0.11,-0.22){2}{\line(0,-1){0.22}}
\multiput(178.53,75.21)(0.12,-0.22){2}{\line(0,-1){0.22}}
\multiput(178.26,75.63)(0.14,-0.21){2}{\line(0,-1){0.21}}
\multiput(177.97,76.04)(0.15,-0.2){2}{\line(0,-1){0.2}}
\multiput(177.66,76.43)(0.1,-0.13){3}{\line(0,-1){0.13}}
\multiput(177.33,76.8)(0.11,-0.12){3}{\line(0,-1){0.12}}
\multiput(176.98,77.16)(0.12,-0.12){3}{\line(0,-1){0.12}}
\multiput(176.62,77.5)(0.12,-0.11){3}{\line(1,0){0.12}}
\multiput(176.23,77.82)(0.13,-0.11){3}{\line(1,0){0.13}}
\multiput(175.84,78.12)(0.13,-0.1){3}{\line(1,0){0.13}}
\multiput(175.43,78.4)(0.21,-0.14){2}{\line(1,0){0.21}}
\multiput(175,78.66)(0.21,-0.13){2}{\line(1,0){0.21}}
\multiput(174.56,78.9)(0.22,-0.12){2}{\line(1,0){0.22}}
\multiput(174.11,79.12)(0.22,-0.11){2}{\line(1,0){0.22}}
\multiput(173.65,79.31)(0.23,-0.1){2}{\line(1,0){0.23}}
\multiput(173.18,79.48)(0.47,-0.17){1}{\line(1,0){0.47}}
\multiput(172.71,79.63)(0.48,-0.15){1}{\line(1,0){0.48}}
\multiput(172.23,79.75)(0.48,-0.12){1}{\line(1,0){0.48}}
\multiput(171.74,79.85)(0.49,-0.1){1}{\line(1,0){0.49}}
\multiput(171.24,79.92)(0.49,-0.07){1}{\line(1,0){0.49}}
\multiput(170.75,79.97)(0.5,-0.05){1}{\line(1,0){0.5}}
\multiput(170.25,80)(0.5,-0.02){1}{\line(1,0){0.5}}
\put(169.75,80){\line(1,0){0.5}}
\multiput(169.25,79.97)(0.5,0.02){1}{\line(1,0){0.5}}
\multiput(168.76,79.92)(0.5,0.05){1}{\line(1,0){0.5}}
\multiput(168.26,79.85)(0.49,0.07){1}{\line(1,0){0.49}}
\multiput(167.77,79.75)(0.49,0.1){1}{\line(1,0){0.49}}
\multiput(167.29,79.63)(0.48,0.12){1}{\line(1,0){0.48}}
\multiput(166.82,79.48)(0.48,0.15){1}{\line(1,0){0.48}}
\multiput(166.35,79.31)(0.47,0.17){1}{\line(1,0){0.47}}
\multiput(165.89,79.12)(0.23,0.1){2}{\line(1,0){0.23}}
\multiput(165.44,78.9)(0.22,0.11){2}{\line(1,0){0.22}}
\multiput(165,78.66)(0.22,0.12){2}{\line(1,0){0.22}}
\multiput(164.57,78.4)(0.21,0.13){2}{\line(1,0){0.21}}
\multiput(164.16,78.12)(0.21,0.14){2}{\line(1,0){0.21}}
\multiput(163.77,77.82)(0.13,0.1){3}{\line(1,0){0.13}}
\multiput(163.38,77.5)(0.13,0.11){3}{\line(1,0){0.13}}
\multiput(163.02,77.16)(0.12,0.11){3}{\line(1,0){0.12}}
\multiput(162.67,76.8)(0.12,0.12){3}{\line(0,1){0.12}}
\multiput(162.34,76.43)(0.11,0.12){3}{\line(0,1){0.12}}
\multiput(162.03,76.04)(0.1,0.13){3}{\line(0,1){0.13}}
\multiput(161.74,75.63)(0.15,0.2){2}{\line(0,1){0.2}}
\multiput(161.47,75.21)(0.14,0.21){2}{\line(0,1){0.21}}
\multiput(161.22,74.78)(0.12,0.22){2}{\line(0,1){0.22}}
\multiput(160.99,74.34)(0.11,0.22){2}{\line(0,1){0.22}}
\multiput(160.79,73.88)(0.1,0.23){2}{\line(0,1){0.23}}
\multiput(160.6,73.42)(0.09,0.23){2}{\line(0,1){0.23}}
\multiput(160.44,72.95)(0.16,0.47){1}{\line(0,1){0.47}}
\multiput(160.31,72.47)(0.14,0.48){1}{\line(0,1){0.48}}
\multiput(160.2,71.98)(0.11,0.49){1}{\line(0,1){0.49}}
\multiput(160.11,71.49)(0.09,0.49){1}{\line(0,1){0.49}}
\multiput(160.05,71)(0.06,0.49){1}{\line(0,1){0.49}}
\multiput(160.01,70.5)(0.04,0.5){1}{\line(0,1){0.5}}
\multiput(160,70)(0.01,0.5){1}{\line(0,1){0.5}}

\linethickness{0.3mm}
\put(160,60){\line(0,1){10}}
\linethickness{0.3mm}
\put(170,25){\line(0,1){25}}
\linethickness{0.3mm}
\put(180,25){\line(0,1){45}}
\linethickness{0.3mm}
\put(30,25){\line(0,1){25}}
\linethickness{0.3mm}
\put(10,25){\line(0,1){25}}
\put(94,77){\makebox(0,0)[cc]{$\<,\>^{-1}$}}

\put(174,86){\makebox(0,0)[cc]{$\<,\>^{-1}$}}

\put(70,40){\makebox(0,0)[cc]{$\partial$}}

\put(14,64){\makebox(0,0)[cc]{$\partial_E$}}

\put(25,82){\makebox(0,0)[cc]{$E$}}

\put(65,85){\makebox(0,0)[cc]{}}

\put(43,55){\makebox(0,0)[cc]{$=$}}

\put(115,55){\makebox(0,0)[cc]{$-$}}

\put(105,64){\makebox(0,0)[cc]{$E$}}

\put(75,65){\makebox(0,0)[cc]{$\overline{E}$}}

\put(154,64){\makebox(0,0)[cc]{$\tilde\partial$}}

\put(140,35){\makebox(0,0)[cc]{$\<,\>$}}

\put(77,47){\makebox(0,0)[cc]{$\<,\>$}}

\put(35,29.5){\makebox(0,0)[cc]{$E$}}

\put(1,30){\makebox(0,0)[cc]{$\Omega^{1,0}$}}

\end{picture}

\end{proposition}
\proof
Note that the RHS of the diagram only depends on the value $\<,\>^{-1}\in \overline{E}\tens_A E$ (with the emphasis on
$\tens_A$). It is reasonably easy to see that this defines a left $\pd$-connection. 
Then applying $\id\tens\<-,\overline{c}\>$ to this shows that
\[
\pd\<e,\overline{c}\>=(\id\tens\<,\>)(\pd_E e\tens \overline{c})+(\<,\>\tens\id)(e\tens\tilde\pd\,\overline{c})\ .
\]
We now need to check applying the 
$\pdol$ derivative, which means checking the equation 

\unitlength 0.5 mm
\begin{picture}(155,55)(-10,23)
\linethickness{0.3mm}
\put(10,60){\line(0,1){15}}
\linethickness{0.3mm}
\put(30,60){\line(0,1){15}}
\linethickness{0.3mm}
\multiput(10,60)(0.01,-0.5){1}{\line(0,-1){0.5}}
\multiput(10.01,59.5)(0.04,-0.5){1}{\line(0,-1){0.5}}
\multiput(10.05,59)(0.06,-0.49){1}{\line(0,-1){0.49}}
\multiput(10.11,58.51)(0.09,-0.49){1}{\line(0,-1){0.49}}
\multiput(10.2,58.02)(0.11,-0.49){1}{\line(0,-1){0.49}}
\multiput(10.31,57.53)(0.14,-0.48){1}{\line(0,-1){0.48}}
\multiput(10.44,57.05)(0.16,-0.47){1}{\line(0,-1){0.47}}
\multiput(10.6,56.58)(0.09,-0.23){2}{\line(0,-1){0.23}}
\multiput(10.79,56.12)(0.1,-0.23){2}{\line(0,-1){0.23}}
\multiput(10.99,55.66)(0.11,-0.22){2}{\line(0,-1){0.22}}
\multiput(11.22,55.22)(0.12,-0.22){2}{\line(0,-1){0.22}}
\multiput(11.47,54.79)(0.14,-0.21){2}{\line(0,-1){0.21}}
\multiput(11.74,54.37)(0.15,-0.2){2}{\line(0,-1){0.2}}
\multiput(12.03,53.96)(0.1,-0.13){3}{\line(0,-1){0.13}}
\multiput(12.34,53.57)(0.11,-0.12){3}{\line(0,-1){0.12}}
\multiput(12.67,53.2)(0.12,-0.12){3}{\line(0,-1){0.12}}
\multiput(13.02,52.84)(0.12,-0.11){3}{\line(1,0){0.12}}
\multiput(13.38,52.5)(0.13,-0.11){3}{\line(1,0){0.13}}
\multiput(13.77,52.18)(0.13,-0.1){3}{\line(1,0){0.13}}
\multiput(14.16,51.88)(0.21,-0.14){2}{\line(1,0){0.21}}
\multiput(14.57,51.6)(0.21,-0.13){2}{\line(1,0){0.21}}
\multiput(15,51.34)(0.22,-0.12){2}{\line(1,0){0.22}}
\multiput(15.44,51.1)(0.22,-0.11){2}{\line(1,0){0.22}}
\multiput(15.89,50.88)(0.23,-0.1){2}{\line(1,0){0.23}}
\multiput(16.35,50.69)(0.47,-0.17){1}{\line(1,0){0.47}}
\multiput(16.82,50.52)(0.48,-0.15){1}{\line(1,0){0.48}}
\multiput(17.29,50.37)(0.48,-0.12){1}{\line(1,0){0.48}}
\multiput(17.77,50.25)(0.49,-0.1){1}{\line(1,0){0.49}}
\multiput(18.26,50.15)(0.49,-0.07){1}{\line(1,0){0.49}}
\multiput(18.76,50.08)(0.5,-0.05){1}{\line(1,0){0.5}}
\multiput(19.25,50.03)(0.5,-0.02){1}{\line(1,0){0.5}}
\put(19.75,50){\line(1,0){0.5}}
\multiput(20.25,50)(0.5,0.02){1}{\line(1,0){0.5}}
\multiput(20.75,50.03)(0.5,0.05){1}{\line(1,0){0.5}}
\multiput(21.24,50.08)(0.49,0.07){1}{\line(1,0){0.49}}
\multiput(21.74,50.15)(0.49,0.1){1}{\line(1,0){0.49}}
\multiput(22.23,50.25)(0.48,0.12){1}{\line(1,0){0.48}}
\multiput(22.71,50.37)(0.48,0.15){1}{\line(1,0){0.48}}
\multiput(23.18,50.52)(0.47,0.17){1}{\line(1,0){0.47}}
\multiput(23.65,50.69)(0.23,0.1){2}{\line(1,0){0.23}}
\multiput(24.11,50.88)(0.22,0.11){2}{\line(1,0){0.22}}
\multiput(24.56,51.1)(0.22,0.12){2}{\line(1,0){0.22}}
\multiput(25,51.34)(0.21,0.13){2}{\line(1,0){0.21}}
\multiput(25.43,51.6)(0.21,0.14){2}{\line(1,0){0.21}}
\multiput(25.84,51.88)(0.13,0.1){3}{\line(1,0){0.13}}
\multiput(26.23,52.18)(0.13,0.11){3}{\line(1,0){0.13}}
\multiput(26.62,52.5)(0.12,0.11){3}{\line(1,0){0.12}}
\multiput(26.98,52.84)(0.12,0.12){3}{\line(0,1){0.12}}
\multiput(27.33,53.2)(0.11,0.12){3}{\line(0,1){0.12}}
\multiput(27.66,53.57)(0.1,0.13){3}{\line(0,1){0.13}}
\multiput(27.97,53.96)(0.15,0.2){2}{\line(0,1){0.2}}
\multiput(28.26,54.37)(0.14,0.21){2}{\line(0,1){0.21}}
\multiput(28.53,54.79)(0.12,0.22){2}{\line(0,1){0.22}}
\multiput(28.78,55.22)(0.11,0.22){2}{\line(0,1){0.22}}
\multiput(29.01,55.66)(0.1,0.23){2}{\line(0,1){0.23}}
\multiput(29.21,56.12)(0.09,0.23){2}{\line(0,1){0.23}}
\multiput(29.4,56.58)(0.16,0.47){1}{\line(0,1){0.47}}
\multiput(29.56,57.05)(0.14,0.48){1}{\line(0,1){0.48}}
\multiput(29.69,57.53)(0.11,0.49){1}{\line(0,1){0.49}}
\multiput(29.8,58.02)(0.09,0.49){1}{\line(0,1){0.49}}
\multiput(29.89,58.51)(0.06,0.49){1}{\line(0,1){0.49}}
\multiput(29.95,59)(0.04,0.5){1}{\line(0,1){0.5}}
\multiput(29.99,59.5)(0.01,0.5){1}{\line(0,1){0.5}}

\linethickness{0.3mm}
\put(20,45){\line(0,1){5}}
\linethickness{0.3mm}
\put(20,25){\line(0,1){10}}
\linethickness{0.3mm}
\put(65,60){\line(0,1){15}}
\linethickness{0.3mm}
\multiput(74.99,50.5)(0.01,-0.5){1}{\line(0,-1){0.5}}
\multiput(74.95,51)(0.04,-0.5){1}{\line(0,-1){0.5}}
\multiput(74.89,51.49)(0.06,-0.49){1}{\line(0,-1){0.49}}
\multiput(74.8,51.98)(0.09,-0.49){1}{\line(0,-1){0.49}}
\multiput(74.69,52.47)(0.11,-0.49){1}{\line(0,-1){0.49}}
\multiput(74.56,52.95)(0.14,-0.48){1}{\line(0,-1){0.48}}
\multiput(74.4,53.42)(0.16,-0.47){1}{\line(0,-1){0.47}}
\multiput(74.21,53.88)(0.09,-0.23){2}{\line(0,-1){0.23}}
\multiput(74.01,54.34)(0.1,-0.23){2}{\line(0,-1){0.23}}
\multiput(73.78,54.78)(0.11,-0.22){2}{\line(0,-1){0.22}}
\multiput(73.53,55.21)(0.12,-0.22){2}{\line(0,-1){0.22}}
\multiput(73.26,55.63)(0.14,-0.21){2}{\line(0,-1){0.21}}
\multiput(72.97,56.04)(0.15,-0.2){2}{\line(0,-1){0.2}}
\multiput(72.66,56.43)(0.1,-0.13){3}{\line(0,-1){0.13}}
\multiput(72.33,56.8)(0.11,-0.12){3}{\line(0,-1){0.12}}
\multiput(71.98,57.16)(0.12,-0.12){3}{\line(0,-1){0.12}}
\multiput(71.62,57.5)(0.12,-0.11){3}{\line(1,0){0.12}}
\multiput(71.23,57.82)(0.13,-0.11){3}{\line(1,0){0.13}}
\multiput(70.84,58.12)(0.13,-0.1){3}{\line(1,0){0.13}}
\multiput(70.43,58.4)(0.21,-0.14){2}{\line(1,0){0.21}}
\multiput(70,58.66)(0.21,-0.13){2}{\line(1,0){0.21}}
\multiput(69.56,58.9)(0.22,-0.12){2}{\line(1,0){0.22}}
\multiput(69.11,59.12)(0.22,-0.11){2}{\line(1,0){0.22}}
\multiput(68.65,59.31)(0.23,-0.1){2}{\line(1,0){0.23}}
\multiput(68.18,59.48)(0.47,-0.17){1}{\line(1,0){0.47}}
\multiput(67.71,59.63)(0.48,-0.15){1}{\line(1,0){0.48}}
\multiput(67.23,59.75)(0.48,-0.12){1}{\line(1,0){0.48}}
\multiput(66.74,59.85)(0.49,-0.1){1}{\line(1,0){0.49}}
\multiput(66.24,59.92)(0.49,-0.07){1}{\line(1,0){0.49}}
\multiput(65.75,59.97)(0.5,-0.05){1}{\line(1,0){0.5}}
\multiput(65.25,60)(0.5,-0.02){1}{\line(1,0){0.5}}
\put(64.75,60){\line(1,0){0.5}}
\multiput(64.25,59.97)(0.5,0.02){1}{\line(1,0){0.5}}
\multiput(63.76,59.92)(0.5,0.05){1}{\line(1,0){0.5}}
\multiput(63.26,59.85)(0.49,0.07){1}{\line(1,0){0.49}}
\multiput(62.77,59.75)(0.49,0.1){1}{\line(1,0){0.49}}
\multiput(62.29,59.63)(0.48,0.12){1}{\line(1,0){0.48}}
\multiput(61.82,59.48)(0.48,0.15){1}{\line(1,0){0.48}}
\multiput(61.35,59.31)(0.47,0.17){1}{\line(1,0){0.47}}
\multiput(60.89,59.12)(0.23,0.1){2}{\line(1,0){0.23}}
\multiput(60.44,58.9)(0.22,0.11){2}{\line(1,0){0.22}}
\multiput(60,58.66)(0.22,0.12){2}{\line(1,0){0.22}}
\multiput(59.57,58.4)(0.21,0.13){2}{\line(1,0){0.21}}
\multiput(59.16,58.12)(0.21,0.14){2}{\line(1,0){0.21}}
\multiput(58.77,57.82)(0.13,0.1){3}{\line(1,0){0.13}}
\multiput(58.38,57.5)(0.13,0.11){3}{\line(1,0){0.13}}
\multiput(58.02,57.16)(0.12,0.11){3}{\line(1,0){0.12}}
\multiput(57.67,56.8)(0.12,0.12){3}{\line(0,1){0.12}}
\multiput(57.34,56.43)(0.11,0.12){3}{\line(0,1){0.12}}
\multiput(57.03,56.04)(0.1,0.13){3}{\line(0,1){0.13}}
\multiput(56.74,55.63)(0.15,0.2){2}{\line(0,1){0.2}}
\multiput(56.47,55.21)(0.14,0.21){2}{\line(0,1){0.21}}
\multiput(56.22,54.78)(0.12,0.22){2}{\line(0,1){0.22}}
\multiput(55.99,54.34)(0.11,0.22){2}{\line(0,1){0.22}}
\multiput(55.79,53.88)(0.1,0.23){2}{\line(0,1){0.23}}
\multiput(55.6,53.42)(0.09,0.23){2}{\line(0,1){0.23}}
\multiput(55.44,52.95)(0.16,0.47){1}{\line(0,1){0.47}}
\multiput(55.31,52.47)(0.14,0.48){1}{\line(0,1){0.48}}
\multiput(55.2,51.98)(0.11,0.49){1}{\line(0,1){0.49}}
\multiput(55.11,51.49)(0.09,0.49){1}{\line(0,1){0.49}}
\multiput(55.05,51)(0.06,0.49){1}{\line(0,1){0.49}}
\multiput(55.01,50.5)(0.04,0.5){1}{\line(0,1){0.5}}
\multiput(55,50)(0.01,0.5){1}{\line(0,1){0.5}}

\linethickness{0.3mm}
\put(95,50){\line(0,1){25}}
\linethickness{0.3mm}
\multiput(75,50)(0.01,-0.5){1}{\line(0,-1){0.5}}
\multiput(75.01,49.5)(0.04,-0.5){1}{\line(0,-1){0.5}}
\multiput(75.05,49)(0.06,-0.49){1}{\line(0,-1){0.49}}
\multiput(75.11,48.51)(0.09,-0.49){1}{\line(0,-1){0.49}}
\multiput(75.2,48.02)(0.11,-0.49){1}{\line(0,-1){0.49}}
\multiput(75.31,47.53)(0.14,-0.48){1}{\line(0,-1){0.48}}
\multiput(75.44,47.05)(0.16,-0.47){1}{\line(0,-1){0.47}}
\multiput(75.6,46.58)(0.09,-0.23){2}{\line(0,-1){0.23}}
\multiput(75.79,46.12)(0.1,-0.23){2}{\line(0,-1){0.23}}
\multiput(75.99,45.66)(0.11,-0.22){2}{\line(0,-1){0.22}}
\multiput(76.22,45.22)(0.12,-0.22){2}{\line(0,-1){0.22}}
\multiput(76.47,44.79)(0.14,-0.21){2}{\line(0,-1){0.21}}
\multiput(76.74,44.37)(0.15,-0.2){2}{\line(0,-1){0.2}}
\multiput(77.03,43.96)(0.1,-0.13){3}{\line(0,-1){0.13}}
\multiput(77.34,43.57)(0.11,-0.12){3}{\line(0,-1){0.12}}
\multiput(77.67,43.2)(0.12,-0.12){3}{\line(0,-1){0.12}}
\multiput(78.02,42.84)(0.12,-0.11){3}{\line(1,0){0.12}}
\multiput(78.38,42.5)(0.13,-0.11){3}{\line(1,0){0.13}}
\multiput(78.77,42.18)(0.13,-0.1){3}{\line(1,0){0.13}}
\multiput(79.16,41.88)(0.21,-0.14){2}{\line(1,0){0.21}}
\multiput(79.57,41.6)(0.21,-0.13){2}{\line(1,0){0.21}}
\multiput(80,41.34)(0.22,-0.12){2}{\line(1,0){0.22}}
\multiput(80.44,41.1)(0.22,-0.11){2}{\line(1,0){0.22}}
\multiput(80.89,40.88)(0.23,-0.1){2}{\line(1,0){0.23}}
\multiput(81.35,40.69)(0.47,-0.17){1}{\line(1,0){0.47}}
\multiput(81.82,40.52)(0.48,-0.15){1}{\line(1,0){0.48}}
\multiput(82.29,40.37)(0.48,-0.12){1}{\line(1,0){0.48}}
\multiput(82.77,40.25)(0.49,-0.1){1}{\line(1,0){0.49}}
\multiput(83.26,40.15)(0.49,-0.07){1}{\line(1,0){0.49}}
\multiput(83.76,40.08)(0.5,-0.05){1}{\line(1,0){0.5}}
\multiput(84.25,40.03)(0.5,-0.02){1}{\line(1,0){0.5}}
\put(84.75,40){\line(1,0){0.5}}
\multiput(85.25,40)(0.5,0.02){1}{\line(1,0){0.5}}
\multiput(85.75,40.03)(0.5,0.05){1}{\line(1,0){0.5}}
\multiput(86.24,40.08)(0.49,0.07){1}{\line(1,0){0.49}}
\multiput(86.74,40.15)(0.49,0.1){1}{\line(1,0){0.49}}
\multiput(87.23,40.25)(0.48,0.12){1}{\line(1,0){0.48}}
\multiput(87.71,40.37)(0.48,0.15){1}{\line(1,0){0.48}}
\multiput(88.18,40.52)(0.47,0.17){1}{\line(1,0){0.47}}
\multiput(88.65,40.69)(0.23,0.1){2}{\line(1,0){0.23}}
\multiput(89.11,40.88)(0.22,0.11){2}{\line(1,0){0.22}}
\multiput(89.56,41.1)(0.22,0.12){2}{\line(1,0){0.22}}
\multiput(90,41.34)(0.21,0.13){2}{\line(1,0){0.21}}
\multiput(90.43,41.6)(0.21,0.14){2}{\line(1,0){0.21}}
\multiput(90.84,41.88)(0.13,0.1){3}{\line(1,0){0.13}}
\multiput(91.23,42.18)(0.13,0.11){3}{\line(1,0){0.13}}
\multiput(91.62,42.5)(0.12,0.11){3}{\line(1,0){0.12}}
\multiput(91.98,42.84)(0.12,0.12){3}{\line(0,1){0.12}}
\multiput(92.33,43.2)(0.11,0.12){3}{\line(0,1){0.12}}
\multiput(92.66,43.57)(0.1,0.13){3}{\line(0,1){0.13}}
\multiput(92.97,43.96)(0.15,0.2){2}{\line(0,1){0.2}}
\multiput(93.26,44.37)(0.14,0.21){2}{\line(0,1){0.21}}
\multiput(93.53,44.79)(0.12,0.22){2}{\line(0,1){0.22}}
\multiput(93.78,45.22)(0.11,0.22){2}{\line(0,1){0.22}}
\multiput(94.01,45.66)(0.1,0.23){2}{\line(0,1){0.23}}
\multiput(94.21,46.12)(0.09,0.23){2}{\line(0,1){0.23}}
\multiput(94.4,46.58)(0.16,0.47){1}{\line(0,1){0.47}}
\multiput(94.56,47.05)(0.14,0.48){1}{\line(0,1){0.48}}
\multiput(94.69,47.53)(0.11,0.49){1}{\line(0,1){0.49}}
\multiput(94.8,48.02)(0.09,0.49){1}{\line(0,1){0.49}}
\multiput(94.89,48.51)(0.06,0.49){1}{\line(0,1){0.49}}
\multiput(94.95,49)(0.04,0.5){1}{\line(0,1){0.5}}
\multiput(94.99,49.5)(0.01,0.5){1}{\line(0,1){0.5}}

\linethickness{0.3mm}
\put(55,25){\line(0,1){25}}
\linethickness{0.3mm}
\put(115,50){\line(0,1){25}}
\linethickness{0.3mm}
\multiput(115,50)(0.01,-0.5){1}{\line(0,-1){0.5}}
\multiput(115.01,49.5)(0.04,-0.5){1}{\line(0,-1){0.5}}
\multiput(115.05,49)(0.06,-0.49){1}{\line(0,-1){0.49}}
\multiput(115.11,48.51)(0.09,-0.49){1}{\line(0,-1){0.49}}
\multiput(115.2,48.02)(0.11,-0.49){1}{\line(0,-1){0.49}}
\multiput(115.31,47.53)(0.14,-0.48){1}{\line(0,-1){0.48}}
\multiput(115.44,47.05)(0.16,-0.47){1}{\line(0,-1){0.47}}
\multiput(115.6,46.58)(0.09,-0.23){2}{\line(0,-1){0.23}}
\multiput(115.79,46.12)(0.1,-0.23){2}{\line(0,-1){0.23}}
\multiput(115.99,45.66)(0.11,-0.22){2}{\line(0,-1){0.22}}
\multiput(116.22,45.22)(0.12,-0.22){2}{\line(0,-1){0.22}}
\multiput(116.47,44.79)(0.14,-0.21){2}{\line(0,-1){0.21}}
\multiput(116.74,44.37)(0.15,-0.2){2}{\line(0,-1){0.2}}
\multiput(117.03,43.96)(0.1,-0.13){3}{\line(0,-1){0.13}}
\multiput(117.34,43.57)(0.11,-0.12){3}{\line(0,-1){0.12}}
\multiput(117.67,43.2)(0.12,-0.12){3}{\line(0,-1){0.12}}
\multiput(118.02,42.84)(0.12,-0.11){3}{\line(1,0){0.12}}
\multiput(118.38,42.5)(0.13,-0.11){3}{\line(1,0){0.13}}
\multiput(118.77,42.18)(0.13,-0.1){3}{\line(1,0){0.13}}
\multiput(119.16,41.88)(0.21,-0.14){2}{\line(1,0){0.21}}
\multiput(119.57,41.6)(0.21,-0.13){2}{\line(1,0){0.21}}
\multiput(120,41.34)(0.22,-0.12){2}{\line(1,0){0.22}}
\multiput(120.44,41.1)(0.22,-0.11){2}{\line(1,0){0.22}}
\multiput(120.89,40.88)(0.23,-0.1){2}{\line(1,0){0.23}}
\multiput(121.35,40.69)(0.47,-0.17){1}{\line(1,0){0.47}}
\multiput(121.82,40.52)(0.48,-0.15){1}{\line(1,0){0.48}}
\multiput(122.29,40.37)(0.48,-0.12){1}{\line(1,0){0.48}}
\multiput(122.77,40.25)(0.49,-0.1){1}{\line(1,0){0.49}}
\multiput(123.26,40.15)(0.49,-0.07){1}{\line(1,0){0.49}}
\multiput(123.76,40.08)(0.5,-0.05){1}{\line(1,0){0.5}}
\multiput(124.25,40.03)(0.5,-0.02){1}{\line(1,0){0.5}}
\put(124.75,40){\line(1,0){0.5}}
\multiput(125.25,40)(0.5,0.02){1}{\line(1,0){0.5}}
\multiput(125.75,40.03)(0.5,0.05){1}{\line(1,0){0.5}}
\multiput(126.24,40.08)(0.49,0.07){1}{\line(1,0){0.49}}
\multiput(126.74,40.15)(0.49,0.1){1}{\line(1,0){0.49}}
\multiput(127.23,40.25)(0.48,0.12){1}{\line(1,0){0.48}}
\multiput(127.71,40.37)(0.48,0.15){1}{\line(1,0){0.48}}
\multiput(128.18,40.52)(0.47,0.17){1}{\line(1,0){0.47}}
\multiput(128.65,40.69)(0.23,0.1){2}{\line(1,0){0.23}}
\multiput(129.11,40.88)(0.22,0.11){2}{\line(1,0){0.22}}
\multiput(129.56,41.1)(0.22,0.12){2}{\line(1,0){0.22}}
\multiput(130,41.34)(0.21,0.13){2}{\line(1,0){0.21}}
\multiput(130.43,41.6)(0.21,0.14){2}{\line(1,0){0.21}}
\multiput(130.84,41.88)(0.13,0.1){3}{\line(1,0){0.13}}
\multiput(131.23,42.18)(0.13,0.11){3}{\line(1,0){0.13}}
\multiput(131.62,42.5)(0.12,0.11){3}{\line(1,0){0.12}}
\multiput(131.98,42.84)(0.12,0.12){3}{\line(0,1){0.12}}
\multiput(132.33,43.2)(0.11,0.12){3}{\line(0,1){0.12}}
\multiput(132.66,43.57)(0.1,0.13){3}{\line(0,1){0.13}}
\multiput(132.97,43.96)(0.15,0.2){2}{\line(0,1){0.2}}
\multiput(133.26,44.37)(0.14,0.21){2}{\line(0,1){0.21}}
\multiput(133.53,44.79)(0.12,0.22){2}{\line(0,1){0.22}}
\multiput(133.78,45.22)(0.11,0.22){2}{\line(0,1){0.22}}
\multiput(134.01,45.66)(0.1,0.23){2}{\line(0,1){0.23}}
\multiput(134.21,46.12)(0.09,0.23){2}{\line(0,1){0.23}}
\multiput(134.4,46.58)(0.16,0.47){1}{\line(0,1){0.47}}
\multiput(134.56,47.05)(0.14,0.48){1}{\line(0,1){0.48}}
\multiput(134.69,47.53)(0.11,0.49){1}{\line(0,1){0.49}}
\multiput(134.8,48.02)(0.09,0.49){1}{\line(0,1){0.49}}
\multiput(134.89,48.51)(0.06,0.49){1}{\line(0,1){0.49}}
\multiput(134.95,49)(0.04,0.5){1}{\line(0,1){0.5}}
\multiput(134.99,49.5)(0.01,0.5){1}{\line(0,1){0.5}}

\linethickness{0.3mm}
\multiput(154.99,50.5)(0.01,-0.5){1}{\line(0,-1){0.5}}
\multiput(154.95,51)(0.04,-0.5){1}{\line(0,-1){0.5}}
\multiput(154.89,51.49)(0.06,-0.49){1}{\line(0,-1){0.49}}
\multiput(154.8,51.98)(0.09,-0.49){1}{\line(0,-1){0.49}}
\multiput(154.69,52.47)(0.11,-0.49){1}{\line(0,-1){0.49}}
\multiput(154.56,52.95)(0.14,-0.48){1}{\line(0,-1){0.48}}
\multiput(154.4,53.42)(0.16,-0.47){1}{\line(0,-1){0.47}}
\multiput(154.21,53.88)(0.09,-0.23){2}{\line(0,-1){0.23}}
\multiput(154.01,54.34)(0.1,-0.23){2}{\line(0,-1){0.23}}
\multiput(153.78,54.78)(0.11,-0.22){2}{\line(0,-1){0.22}}
\multiput(153.53,55.21)(0.12,-0.22){2}{\line(0,-1){0.22}}
\multiput(153.26,55.63)(0.14,-0.21){2}{\line(0,-1){0.21}}
\multiput(152.97,56.04)(0.15,-0.2){2}{\line(0,-1){0.2}}
\multiput(152.66,56.43)(0.1,-0.13){3}{\line(0,-1){0.13}}
\multiput(152.33,56.8)(0.11,-0.12){3}{\line(0,-1){0.12}}
\multiput(151.98,57.16)(0.12,-0.12){3}{\line(0,-1){0.12}}
\multiput(151.62,57.5)(0.12,-0.11){3}{\line(1,0){0.12}}
\multiput(151.23,57.82)(0.13,-0.11){3}{\line(1,0){0.13}}
\multiput(150.84,58.12)(0.13,-0.1){3}{\line(1,0){0.13}}
\multiput(150.43,58.4)(0.21,-0.14){2}{\line(1,0){0.21}}
\multiput(150,58.66)(0.21,-0.13){2}{\line(1,0){0.21}}
\multiput(149.56,58.9)(0.22,-0.12){2}{\line(1,0){0.22}}
\multiput(149.11,59.12)(0.22,-0.11){2}{\line(1,0){0.22}}
\multiput(148.65,59.31)(0.23,-0.1){2}{\line(1,0){0.23}}
\multiput(148.18,59.48)(0.47,-0.17){1}{\line(1,0){0.47}}
\multiput(147.71,59.63)(0.48,-0.15){1}{\line(1,0){0.48}}
\multiput(147.23,59.75)(0.48,-0.12){1}{\line(1,0){0.48}}
\multiput(146.74,59.85)(0.49,-0.1){1}{\line(1,0){0.49}}
\multiput(146.24,59.92)(0.49,-0.07){1}{\line(1,0){0.49}}
\multiput(145.75,59.97)(0.5,-0.05){1}{\line(1,0){0.5}}
\multiput(145.25,60)(0.5,-0.02){1}{\line(1,0){0.5}}
\put(144.75,60){\line(1,0){0.5}}
\multiput(144.25,59.97)(0.5,0.02){1}{\line(1,0){0.5}}
\multiput(143.76,59.92)(0.5,0.05){1}{\line(1,0){0.5}}
\multiput(143.26,59.85)(0.49,0.07){1}{\line(1,0){0.49}}
\multiput(142.77,59.75)(0.49,0.1){1}{\line(1,0){0.49}}
\multiput(142.29,59.63)(0.48,0.12){1}{\line(1,0){0.48}}
\multiput(141.82,59.48)(0.48,0.15){1}{\line(1,0){0.48}}
\multiput(141.35,59.31)(0.47,0.17){1}{\line(1,0){0.47}}
\multiput(140.89,59.12)(0.23,0.1){2}{\line(1,0){0.23}}
\multiput(140.44,58.9)(0.22,0.11){2}{\line(1,0){0.22}}
\multiput(140,58.66)(0.22,0.12){2}{\line(1,0){0.22}}
\multiput(139.57,58.4)(0.21,0.13){2}{\line(1,0){0.21}}
\multiput(139.16,58.12)(0.21,0.14){2}{\line(1,0){0.21}}
\multiput(138.77,57.82)(0.13,0.1){3}{\line(1,0){0.13}}
\multiput(138.38,57.5)(0.13,0.11){3}{\line(1,0){0.13}}
\multiput(138.02,57.16)(0.12,0.11){3}{\line(1,0){0.12}}
\multiput(137.67,56.8)(0.12,0.12){3}{\line(0,1){0.12}}
\multiput(137.34,56.43)(0.11,0.12){3}{\line(0,1){0.12}}
\multiput(137.03,56.04)(0.1,0.13){3}{\line(0,1){0.13}}
\multiput(136.74,55.63)(0.15,0.2){2}{\line(0,1){0.2}}
\multiput(136.47,55.21)(0.14,0.21){2}{\line(0,1){0.21}}
\multiput(136.22,54.78)(0.12,0.22){2}{\line(0,1){0.22}}
\multiput(135.99,54.34)(0.11,0.22){2}{\line(0,1){0.22}}
\multiput(135.79,53.88)(0.1,0.23){2}{\line(0,1){0.23}}
\multiput(135.6,53.42)(0.09,0.23){2}{\line(0,1){0.23}}
\multiput(135.44,52.95)(0.16,0.47){1}{\line(0,1){0.47}}
\multiput(135.31,52.47)(0.14,0.48){1}{\line(0,1){0.48}}
\multiput(135.2,51.98)(0.11,0.49){1}{\line(0,1){0.49}}
\multiput(135.11,51.49)(0.09,0.49){1}{\line(0,1){0.49}}
\multiput(135.05,51)(0.06,0.49){1}{\line(0,1){0.49}}
\multiput(135.01,50.5)(0.04,0.5){1}{\line(0,1){0.5}}
\multiput(135,50)(0.01,0.5){1}{\line(0,1){0.5}}

\linethickness{0.3mm}
\put(145,60){\line(0,1){15}}
\linethickness{0.3mm}
\put(155,25){\line(0,1){25}}
\put(19,40){\makebox(0,0)[cc]{$\pdol$}}

\put(58,65){\makebox(0,0)[cc]{$\pdol_E$}}

\put(85,35){\makebox(0,0)[cc]{$\<,\>$}}

\put(125,35){\makebox(0,0)[cc]{$\<,\>$}}

\put(27,47){\makebox(0,0)[cc]{$\<,\>$}}

\put(150,65){\makebox(0,0)[cc]{$\hat\pd$}}

\put(15,74.5){\makebox(0,0)[cc]{$E$}}

\put(35,75){\makebox(0,0)[cc]{$\overline{E}$}}

\put(45,50){\makebox(0,0)[cc]{$=$}}

\put(105,50){\makebox(0,0)[cc]{$+$}}

\end{picture}

\noindent where  $\hat\pd: \overline{E}\to
\overline{E}\tens_A \Omega^{0,1}$ is the right $\pdol$-covariant derivative defined by $\hat\pd(\overline{c})=\overline{g}\tens \xi^*$, where $\xi\tens g=\pd_E(c)$, as defined in the previous picture.
We have
\begin{align*}
(\pdol\<e,\overline{c}\>)^*=&\ \pd\<c,\overline{e}\>=(\id\tens\<,\>)(\pd_E c\tens \overline{e})+(\<,\>\tens\id)(c\tens \tilde\pd(\overline{e})) 
= \xi\,\<g,\overline{e}\>+\<c,\overline{f}\>\, \kappa^*\ ,
\end{align*}
so on taking star again,
\begin{align*}
\pdol\<e,\overline{c}\>
=&\ \<g,\overline{e}\>^*\,\xi^*+\kappa\,\<c,\overline{f}\>^* =\<e,\overline{g}\>\,\xi^*+\kappa\,\<f,\overline{c}\>\cr
=&\ (\<,\>\tens\id)(e\tens \hat\pd(\overline{c}))+(\<,\>\tens\id)(\pdol_E(e)\tens \overline{c})\ .\qquad\largesquare
\end{align*}

\begin{proposition}
Suppose that the conditions for Proposition~\ref{unichernnog} hold, and that $E$ is a bimodule with 
 $(E,\pdol_E,\sigma_E)$ a left bimodule $\pdol$-connection on $E$ for $\sigma_E:E\tens_A\Omega^{0,1}A\to 
\Omega^{0,1}A\tens_A E$. Then the connection $\nabla_E$ in Proposition~\ref{unichernnog} is a bimodule connection, where
$\sigma_E:E\tens_A\Omega^{1,0}A\to 
\Omega^{1,0}A\tens_A E$ is defined by,  for $\eta\in\Omega^{1,0}A$
\begin{align*}
\sigma_E(e\tens\eta) =&\ (\<,\>\tens\id\tens\id)\big(\id\tens
(\id\tens \star^{-1})\Upsilon\,\overline{\sigma_E}\, \Upsilon^{-1}(\star\tens\id)
\tens\id\big)(e\tens\eta\tens\<,\>^{-1})
\end{align*}
or maybe more obviously as,
 for $\<,\>^{-1}=\overline{c}\tens g$, and where $\xi\tens k=\sigma_E(c\tens\eta^*)$
\begin{align*}
\sigma_E(e\tens\eta) =&\ \<e,\overline{k}\>\,\xi^*\tens g\ .
\end{align*}
\end{proposition}
\begin{proof} 
From the diagram in Proposition~\ref{unichernnog}, where $\<,\>^{-1}(1)=\overline{c}\tens g$,
\begin{align*}
\pd_E(e.a) =&\ \pd\<e\,a\tens\overline{c}\>\tens g
-(\<,\>\tens\id\tens\id)(e\,a\tens \tilde\pd(\overline{c})\tens g) \cr
=&\ \pd\<e\tens a\,\overline{c}\>\tens g
-(\<,\>\tens\id\tens\id)(e\tens a\,\tilde \pd(\overline{c})\tens g) \cr
=&\ \pd\<e\tens a\,\overline{c}\>\tens g
-(\<,\>\tens\id\tens\id)(e\tens \tilde \pd(a\,\overline{c})\tens g) \cr
&\ +(\<,\>\tens\id\tens\id)(e\tens \tilde \pd(a\, \overline{c})\tens g)
-(\<,\>\tens\id\tens\id)(e\tens a\,\tilde \pd(\overline{c})\tens g)\ .
\end{align*}
Now $a.\overline{c}\tens g=\overline{c}\tens g.a\in \overline{E}\tens_A E$, so we have
\begin{align*}
\pd_E(e.a) 
=&\ \pd\<e\tens \overline{c}\>\tens g\,a
-(\<,\>\tens\id\tens\id)(e\tens \tilde \pd(\overline{c})\tens g\,a) \cr
&\ +(\<,\>\tens\id\tens\id)(e\tens \tilde \pd(a\, \overline{c})\tens g)
-(\<,\>\tens\id\tens\id)(e\tens a\,\tilde \pd(\overline{c})\tens g)\ ,
\end{align*}
so we have
\begin{align*}
\pd_E(e.a)-\pd_E(e).a =&\ 
(\<,\>\tens\id)\big(e\tens ( \tilde \pd(\overline{c.a^*})-a\,\tilde \pd(\overline{c}))\big)\tens g\ .
\end{align*}
Now $\pdol_E(c.a^*)=\pdol_E(c).a^*+\sigma(c\tens\pdol a^*)$, and if we put 
$\kappa\tens f=\pdol_E(c)$ and $\xi\tens k=\sigma_E(c\tens\pdol a^*)$ then
\begin{align*}
 \tilde \pd(\overline{c.a^*})-a\,\tilde \pd(\overline{c}) =&\ \overline{f.a^*}\tens\kappa^*
 + \overline{k}\tens\xi^*-a\,\overline{f}\tens\kappa^* 
 = \overline{k}\tens\xi^*\ ,
\end{align*}
giving
\begin{align*}
\pd_E(e.a)-\pd_E(e).a =&\ 
\<e,\overline{k}\>.\xi^*\tens g\ .
\end{align*}
Finally we use
\begin{align*}
(\id\tens \star^{-1})\Upsilon\,\overline{\sigma_E}\, \Upsilon^{-1}(\star\tens\id)(\pd a\tens\overline{c})
=\overline{k}\tens\xi^*\ ,
\end{align*}
which gives the result.
\end{proof}

\subsection{Christoffel symbol approach}
We shall use the matrix formalism for finitely generated projective modules and results given in \cite{BegMa3}. Of course the use of projection matrices for 
finitely generated projective modules is long established, but the use of matrices for inner products and Christoffel symbols in noncommutative geometry is more recent. 

Suppose that $E$ is a left finitely generated projective module, and fix a dual basis $e^i\in E$ and $e_i\in E^\circ$ for $1\le i\le n$, where $E^\circ=
{}_A\Hom(E,A)$. Then $P_{ji}=e_i(e^j)=\ev(e^j\tens e_i)$ is a matrix with entries in $A$, and $P^2=P$. 
We can describe a left covariant derivative $\nabla_E$ by 
Christoffel symbols, defined by
\begin{eqnarray} \label{chrsymbdef}
\Gamma_k^i = -\,(\id\tens\ev)(\nabla_E e^i \tens e_k)\in\Omega^1 \ ,
\end{eqnarray}
so we have
\begin{eqnarray}\label{christoffeldef}
\nabla_E e^i \,=\, -\,\Gamma_k^i\tens e^k\ .
\end{eqnarray}
Fit the Christoffel symbols
into matrix notation by setting
\begin{eqnarray}\label{chrmatrix}
(\Gamma)_{ij}\,=\,\Gamma^i_j\ .
\end{eqnarray}
Then a necessary and sufficient condition that $\Gamma\in M_n(\Omega^1)$ is the Christoffel symbols for a left connection on $E$ is that
\begin{eqnarray*}
\Gamma\,P\,=\,\Gamma\ ,\quad \Gamma\,=\,P\,\Gamma-\extd P.P\ .
\end{eqnarray*}
The curvature of the connection is given by
\[
R_E(e^i)=-((\extd\Gamma+\Gamma\wedge\Gamma).P)_{ik}\tens e^k\ .
\]

Suppose that we set $g^{ij}=\<e^i,\overline{e^j}\>$, so the hermitian condition gives
$g^{ij*}=g^{ji}$. This corresponds to the invertible bimodule map
$G:\overline{E}\to E^\circ$ being
 $G(\overline{e^i})=e_j.g^{ji}$, and we write the inverse as 
$G^{-1}(e_i)=\overline{g_{ij}.e^j}$, where without
loss of generality
we can assume that $g_{ij}.\ev(e^j\tens e_k)= g_{ik}$.
It is convenient to define matrices $g^\bullet,g_\bullet\in M_n(A)$ by 
\begin{eqnarray}\label{matrixform}
 (g_\bullet)_{ij}=g_{ij}\ ,\quad (g^\bullet)_{ij}=g^{ij}\ .
\end{eqnarray}
and then we have
\begin{eqnarray}\label{matrixform77}
g^{\bullet*}=g^\bullet \ ,\quad g_\bullet^*=g_\bullet\ ,\quad
g^\bullet g_\bullet=P\ ,\quad g_\bullet P=g_\bullet\ ,\quad Pg^\bullet= g^\bullet\ .
\end{eqnarray}

\begin{proposition} \label{unichern}
Given a fgp holomorphic left $A$-module $E$ with a hermitian metric $\<,\>:E\tens \overline{E}\to A$, there is a unique connection
$\nabla_E:E\to\Omega^1 A\tens_A E$, called the Chern connection, which preserves the hermitian metric and for which
$(\pi^{0,1}\tens\id)\nabla_E=\pdol_E$, the canonical $\pdol$-operator for $E$. If we write $\Gamma=\Gamma_++\Gamma_-$, with $\Gamma_+\in M_n(\Omega^{1,0})$ and 
$\Gamma_-\in M_n(\Omega^{0,1})$, then $\Gamma_-$ is determined by $\pdol_E$ and
\begin{align*} 
-\,\Gamma_+= \partial g^\bullet.g_\bullet+g^\bullet\ (\Gamma_-)^* \, g_\bullet\ .
\end{align*}
\end{proposition}
\proof Take a dual basis $(e^i,e_i)$ for $E$, and set $\nabla_E(e^i)=-\,\Gamma^i_{a}\tens e^a$, using the definition of the Christoffel symbols in (\ref{chrsymbdef}). Then the equation (\ref{hermcon1}) for preserving the metric 
evaluated at $e^i\tens \overline{e^j}$ becomes
\begin{eqnarray*}
\extd\,g^{ij} &=& -\,(\id\tens\<,\>)(\Gamma^i_{a}\tens e^a\tens \overline{e^j})-(\<,\>\tens\id)(e^i\tens \overline{e^a}\tens (\Gamma^j_{a})^*)\cr
&=& -\,\Gamma^i_{a}\,g^{aj}-g^{ia}\,(\Gamma^j_{a})^*\ .
\end{eqnarray*}
 Applying $\pi^{1,0}$ to this gives
\begin{eqnarray*}
\pi^{1,0}(\extd g^{ij}) 
&=& -\,\pi^{1,0}(\Gamma^i_{a})\,g^{aj}-g^{ia}\,\pi^{1,0}((\Gamma^j_{a})^*) \cr
&=&  -\,\pi^{1,0}(\Gamma^i_{a})\,g^{aj}-g^{ia}\,(\pi^{0,1}(\Gamma^j_{a}))^* \ ,
\end{eqnarray*}
which is rearranged to give the answer. \qquad$\largesquare$

\begin{proposition}  
The $\Omega^{0,2}$ and $\Omega^{2,0}$ components of the curvature of the Chern connection in Proposition~\ref{unichern}  vanish.
\end{proposition}
\proof
First
\begin{align*}
R_E(e^i)=&\ (\extd\tens\id_E-\id\wedge\nabla_E)\nabla_E(e^i) = -\,\extd\Gamma^i{}_{a}\tens e^a + \Gamma^i{}_{a}\wedge\nabla_E e^a \cr
=&\ -\,\extd\Gamma^i{}_{a}\tens e^a - \Gamma^i{}_{a}\wedge\Gamma^a{}_{b}\tens e^b = -\,(  \extd\Gamma^i{}_{a} + \Gamma^i{}_{j}\wedge\Gamma^j{}_{a}  )P_{ab}\tens e^b\ .
\end{align*}
The $\Omega^{0,2}$ component of the curvature vanishes because $\pdol_E$ has zero holomorphic curvature. 
From Proposition~\ref{unichern}, 
\begin{align*}
\pd \Gamma_+ =&\  \partial g^\bullet\wedge\partial g_\bullet-\partial g^\bullet\wedge \Gamma_-{}^* \, g_\bullet
-g^\bullet\ \partial (\Gamma_-{}^*) \, g_\bullet   +   g^\bullet\ \Gamma_-{}^* \wedge \partial g_\bullet\ ,\cr
\Gamma_+ \wedge \Gamma_+ =&\ \partial g^\bullet.g_\bullet\wedge \partial g^\bullet.g_\bullet+
\partial g^\bullet \wedge  \Gamma_-{}^* \, g_\bullet
+  g^\bullet\ \Gamma_-{}^* \, g_\bullet   \wedge \partial g^\bullet.g_\bullet
+ g^\bullet\ \Gamma_-{}^* \wedge \Gamma_-{}^* \, g_\bullet\ ,
\end{align*}
and then
\begin{align*}
\pd \Gamma_+ + \Gamma_+ \wedge \Gamma_+=&\ (\pd g^\bullet+ g^\bullet\, \Gamma_-{}^*)\wedge(\pd g_\bullet+ g_\bullet.\pd g^\bullet. g_\bullet) -g^\bullet\ \partial (\Gamma_-{}^*) \, g_\bullet
+ g^\bullet\ \Gamma_-{}^* \wedge \Gamma_-{}^* \, g_\bullet \cr
=&\  (\pd g^\bullet+ g^\bullet\, \Gamma_-{}^*)\wedge(\pd g_\bullet+ g_\bullet.\pd g^\bullet. g_\bullet) - g^\bullet\ (\pdol \Gamma_-+\Gamma_- \wedge \Gamma_-)^* \, g_\bullet\ ,
\end{align*}
and the last bracket vanishes as it is just the holomorphic curvature of the holomorphic connection. 
Then
\begin{align*}
(\pd g_\bullet+ g_\bullet.\pd g^\bullet. g_\bullet).P = \pd(P^*).g_\bullet
\end{align*}
and
\begin{align*}
(\pd g^\bullet+ g^\bullet\, \Gamma_-{}^*).P^* =\pd g^\bullet+ g^\bullet\, \Gamma_-{}^*\ ,
\end{align*}
so
\begin{align*}
(\pd \Gamma_+ + \Gamma_+ \wedge \Gamma_+).P
=&\  (\pd g^\bullet+ g^\bullet\, \Gamma_-{}^*).P^* \wedge \pd(P^*).g_\bullet
\end{align*}
where $Q=P^*=g_\bullet g^\bullet$ obeys $Q^2=Q$. 
Differentiating this gives $\pd Q.Q=(1-Q).\pd Q$, so $Q.\pd Q.Q=0$.\qquad$\largesquare$

\section{Examples of Chern connections}

\subsection{The Chern connection on $\Omega^{1,0}$ for $M_2(\C)$}
For the algebra $A=M_2(\C)$, the decomposition $\Omega^1=\Omega^{1,0}\oplus \Omega^{0,1}$ of the differential calculus
gives an integrable almost complex structure. The complex differentials are given by
the graded commutators $\pd=[E_{12}s,-\}$ and $\pdol=[E_{21}t,-\}$.
There is a Riemannian structure 
\[
\<u\oplus v,\overline{x\oplus y}\> = ux^*+vy^*\in A\ ,
\]
and this can be converted to a Hilbert space inner product by taking the trace.

Now put a holomorphic structure on the bimodule $E=\Omega^{1,0}$. For the holomorphic connection $\pdol_E$ use 
\[
\Omega^{1,0} \stackrel{\pdol}\longrightarrow \Omega^{1,1} \stackrel{\wedge^{-1}}\longrightarrow 
\Omega^{0,1}\tens_{M_2(\C)} \Omega^{1,0}\ ,
\]
where we have used the fact that $\wedge:\Omega^{0,1}\tens_{M_2(\C)} \Omega^{1,0}\to \Omega^{1,1} $ is a bimodule isomorphism. 
To check that this is a left $\pdol$-connection, for $\xi\in \Omega^{1,0}$,
\[
\pdol_E(a.\xi)=\wedge^{-1}(\pdol a\wedge\xi+a.\pdol\xi)=\pdol a\tens\xi+a.\pdol_E(\xi)\ .
\]
The curvature of this $\pdol$-connection maps to $\Omega^{0,2}\tens_{M_2(\C)} \Omega^{1,0}$, and thus must be zero as we set $s^2=t^2=0$, thus we have exhibited a holomorphic structure on $E=\Omega^{1,0}$. 

We take the single basis element $s$ on $E=\Omega^{1,0}$. 
As $\pdol_E(s)=2\,E_{21}t\tens s$ we get the Christoffel symbol $\Gamma_{-1,1}=-2\,E_{21}t$. We take the metric
\[
\<b\,s,\overline{a\,s}\>=b\,a^*\ ,
\]
so $g^\bullet$ is a 1 by 1 matrix with the single element $g^{1,1}=1$. 
Then by Proposition~\ref{unichern}
\[
\Gamma_{+1,1}=-\,(-2\,E_{21}t)^*=-2\,E_{12}s\ .
\]
so we have
\[
\nabla_E(s)=2\,E_{12}s\tens s + 2\,E_{21}t\tens s \ .
\]
We get a bimodule covariant derivative, as
\[
\nabla_E(s.a)-\nabla_E(s).a=\nabla_E(a.s)-\nabla_E(s).a=\extd a\tens s+[a,\nabla_E(s)]=-\extd a\tens s\ ,
\]
which extends to the map
\[
\sigma_E(a.s\tens\xi)=-a.\xi\tens s\ .
\]
This gives a natural bimodule connection on $\Omega^{1,0}$ which classically on a Kahler manifold would be part of the Levi-Civita connection for the hermitian metric. Note that in Section~\ref{SecMatrix} we have coincidentally taken $S=\Omega^{1,0}\oplus\Omega^{0,1}$ as a bundle but in this example the Chern connection is not the one used there to construct the $\dirac$ operator.

\subsection{Chern connection on the standard quantum sphere}
On $A=\C_q[S^2]$ take the holomorphic connection on $\CS_+$ (generated by $f^+$ as given earlier) given by $\pdol_{\CS_+}:\CS_+\to \Omega^{0,1}\tens_A \CS_+$, where
\[
\pdol_{\CS_+}(x\,f^+)=\pdol x.k_1\tens k_2.f^+\ ,
\]
where $k=a\tens d-q^{-1}\,c\tens b=k_1\tens k_2$ in a compact notation with summation understood. 
(Note $\pdol x$ denotes taking the $e^-$ component of $\pi\extd$.)
This has zero curvature as $\Omega^{0,2}=0$ and we are in the  case where the grading operator $\gamma$ splits the spinor bundle into two parts, one of which is holomorphic. Now we use (\ref{sideswitch}) to switch the side of the antilinearity on the inner product in Section~\ref{ncHopfDirac} as follows:
 \begin{align*}
\<x.f^+,\overline{y.f^+}\> =&\ \<\overline{\J( x.f^+) } , \J^{-1}(y.f^+)\>=\<\delta\,\overline{x^*.f^- } , \delta\,y^*.f^-\>
=\delta^2\mu\, x\,y^*\ .
\end{align*}
For $\dirac$ to be hermitian, Proposition~\ref{prrp1} gives $\delta^2\mu=q$, but as in fact we will only be interested in the inner product on $\CS_+$, we are free to absorb this $q$ factor in the normalisation of the inner product and hence we omit it in what follows. 
Then 
we can write $\<,\>^{-1}(1)=\overline{k_1^*.f^+}\tens k_2.f^+$. 
Then the formula in Proposition~\ref{unichernnog} gives
\begin{align*}
\pd_{\CS_+}(x\,f^+) =&\ \pd\<x\,f^+, \overline{k_1^*.f^+}\>\tens k_2.f^+ -
( \<,\>\tens\id\tens\id )  (x\,f^+\tens \tilde\pd(\overline{k_1^*.f^+}) \tens k_2.f^+)
\end{align*}
Now, where $k_1'\tens k_2'$ is an independent copy of $k$,
\[
\pdol_{\CS_+}(k_1^*.f^+)=\pdol k_1^*.k_1'\tens k_2'.f^+
\]
so
\begin{align*}
\pd_{\CS_+}(x\,f^+) =&\ \pd\<x\,f^+\tens \overline{k_1^*.f^+}\>\tens k_2.f^+ -
 \<x\,f^+,\overline{k_2'.f^+}\>\,(\pdol k_1^*.k_1')^*   \tens k_2.f^+   \cr
 =&\ \pd(x\,k_1)\tens k_2.f^+ -
 x\,(k_2')^*\,(k_1')^*\,\pd k_1   \tens k_2.f^+  = \pd x.k_1\tens k_2.f^+ \ .
\end{align*}
Thus the Chern connection is just
\[
\nabla_{\CS_+}(x\,f^+) =\pi\extd x.k_1\tens k_2.f^+ \ ,
\]
which is just the connection on $\CS_+$ given in Section~\ref{ncHopfDirac}.

We also compute the example $E=\Omega^{1,0}$. We have $\pdol(x\,e^+)=\pdol x\wedge e^+$, and using the isomorphism $\wedge:\Omega^{0,1}\tens_A \Omega^{1,0}\to
\Omega^{1,1}$ we write a holomorphic connection on $E$ as, where $k_1'\tens k_2'$ is another copy of $k_1\tens k_2$ (similarly for further primes)
\[
\pdol_E(x\,e^+) = \pdol x.k_1k_1' \tens k_2'k_2\,e^+\ .
\]
We take the inner product
\[
\<x\,e^+,\overline{y\,e^+}\>_E= x\,y^*\ ,
\]
and then $\<,\>_E^{-1}=\overline{k_1^{\prime*}k_1^*\,e^+} \tens k_2'k_2\,e^+$. Now using
\[
\pdol_E(k_1^{\prime*}k_1^*\,e^+) = \pdol (k_1^{\prime*}k_1^*).k_1''k_1''' \tens k_2'''k_2''\,e^+\ .
\]
we get
\begin{align*}
(\tilde\pd\tens\id)\<,\>_E^{-1} =&\  \overline{  k_2'''k_2''\,e^+  }    \tens (\pdol (k_1^{\prime*}k_1^*).k_1''k_1''' )^*        \tens k_2'k_2\,e^+\ .
\end{align*}
We use the diagram formula in Proposition~\ref{unichernnog} to write
\begin{align*}
\pd_E(x\,e^+) =&\  \pd\< x\,e^+, \overline{k_1^{\prime*}k_1^*\,e^+} \ \>_E \tens k_2'k_2\,e^+ 
-    \< x\,e^+, \overline{  k_2'''k_2''\,e^+  }    \>\, (\pdol (k_1^{\prime*}k_1^*).k_1''k_1''' )^*        \tens k_2'k_2\,e^+ \cr
=&\  \pd(x\,k_1k_1') \tens k_2'k_2\,e^+ 
-    x\,( k_2'''k_2'')^*  \, (k_1''k_1''' )^*   (\pdol (k_1^{\prime*}k_1^*) )^*        \tens k_2'k_2\,e^+ \cr
=&\  \pd(x\,k_1k_1') \tens k_2'k_2\,e^+ 
-  x\, \pd(k_1k_1')      \tens k_2'k_2\,e^+ = \pd(x).k_1k_1'\tens k_2'k_2\,e^+ \ .
\end{align*}
Putting these parts together, we get
\[
\nabla_E(x\,e^+) = \pi\extd x.k_1k_1'\tens k_2'k_2\,e^+\ ,
\]
which, given that the relevant component of the metric in \cite{Ma:rieq} can be written as $g_{-+}=e^-k_1k_2'\tens k_2'k_2 e^+$, is the same as the $\Omega^{1,0}$ part of the quantum Levi-Civita connection on the $q$-sphere found there. Thus, both the connection for the Dirac operator and the Levi-Civita connection for the $q$-sphere are obtained from the Chern construction.

\subsection{A Chern connection on the quantum open disk}
The calculus on the quantum disk in Section~\ref{diskdir} was constructed to be 
$U_q(sl_2)$ invariant, and also carries a hint of the hyperbolic structure. If we look for a central quantum symmetric metric, we are naturally led to  
\[
g=\Y^{-2}\,(\extd z\tens \extd \bar z+q^{-2}\, \extd \bar z\tens \extd z)
\]
Of course, this cannot be regarded as being a Riemannian structure for the closed disk, but rather it lives on the open disk, in fact if we let $q\to 1$ we get the classical hyperbolic metric. The inverse of $w$ appearing here indicates that we are dealing with unbounded functions on the disk. Now for $n\ge 1$,
\[
q^{1/2}X_+\la(\Y^n\Y^{-n})=0
\]
so
\[
q^{1/2}X_+\la(\Y^{-n})=q^{-1}\Y^{-n}\bar z[n]_{q^{-2}}=q^{2n-1}\bar z[n]_{q^{-2}}\Y^{-n}\ .
\]
Now we check that the metric is invariant to the $U_q(sl_2)$ action. This is easy for $q^{H\over 2}$ while for $X_+$ ($X_-$ is similar), 
\begin{align*}
& q^{1/2}X_+\la(\Y^{-2}\,(\extd z\tens \extd \bar z+q^{-2}\, \extd \bar z\tens \extd z)) \cr
=&\ (q^{1/2}X_+\la(\Y^{-2}))\,(   q^{H\over 2}  \la (\extd z\tens \extd \bar z+q^{-2}\, \extd \bar z\tens \extd z)) \cr
& +\ (q^{-H\over 2} \la(\Y^{-2}))\,(   X_+ \la (\extd z\tens \extd \bar z+q^{-2}\, \extd \bar z\tens \extd z)) \cr
=&\ q^{3}\bar z[2]_{q^{-2}}\Y^{-2}(\extd z\tens \extd \bar z+q^{-2}\, \extd \bar z\tens \extd z)\cr
& +\ \Y^{-2}\,(   q^{-H\over 2} \la\extd z\tens X_+ \la  \extd \bar z+q^{-2}\,X_+ \la  \extd \bar z\tens q^{H\over 2} \la \extd z)) \cr
=&\ q^{3}\bar z[2]_{q^{-2}}\Y^{-2}(\extd z\tens \extd \bar z+q^{-2}\, \extd \bar z\tens \extd z)\cr
& -\ \Y^{-2}\,(1+q^{-2})\,(   q\,\extd z\tens \bar z. \extd \bar z +q^{-3}\, \bar z. \extd \bar z\tens \extd z)) \cr
=&\ q^{3}\bar z\, \Y^{-2}([2]_{q^{-2}}  -1-q^{-2}  )(\extd z\tens \extd \bar z+q^{-2}\, \extd \bar z\tens \extd z)=0\ .
\end{align*}

An integrable almost complex structure is given by $J(\extd z)=\mathrm{i}\, \extd z$ and 
$J(\extd \bar z)=-\mathrm{i}\, \extd \bar z$. 
Now we examine the Chern connection for the holomorphic bundle $E=\Omega^{1,0}$ (we use $E$ as $\Omega^{1,0}$ is a rather cumbersome subscript) on the unit disk $\C_q[D]$. For the holomorphic connection $\pdol_E$ (and thus the holomorphic structure) on $E=\Omega^{1,0}$ use 
\[
\Omega^{1,0} \stackrel{\pdol}\longrightarrow \Omega^{1,1} \stackrel{\wedge^{-1}}\longrightarrow 
\Omega^{0,1}\tens_{\C_q[D]} \Omega^{1,0}\ ,
\]
where we have used the fact that $\wedge:\Omega^{0,1}\tens_{\C_q[D]} \Omega^{1,0}\to \Omega^{1,1} $ is a bimodule isomorphism. 
To check that this is a left $\pdol$-connection, for $\xi\in \Omega^{1,0}$,
\[
\pdol_E(a.\xi)=\wedge^{-1}(\pdol a\wedge\xi+a.\pdol\xi)=\pdol a\tens\xi+a.\pdol_E(\xi)\ .
\]
The curvature of this $\pdol$-connection maps to $\Omega^{0,2}\tens_{\C_q[D]} \Omega^{1,0}$, and thus must be zero. 

We take the single basis element $\extd z$ on $E=\Omega^{1,0}$. 
As $\pdol_E(\extd z)=0$ we get $\Gamma_- =0$. Taking the invariant metric, $g^\bullet$ is a 1 by 1 matrix with the single element
\[
g^{1,1}=\<\extd z,\overline{\extd z}\>=\Y^2\ .
\]
Then 
\[
\Gamma_{+1,1}=-\,\pd(\Y^2)\,\Y^{-2}=\bar z\,\extd z\,[2]_{q^{-2}}\Y^{-1}\ ,
\]
so we have the Chern connection
\[
\nabla_E(\extd z)=-\bar z\,\extd z\,[2]_{q^{-2}}\Y^{-1}\tens \extd z\ 
\]
associated to the above hermitian metric. 

Now we shall consider the different bimodule, the sub-bimodule $S_+$ of the spinor bundle from Section~\ref{diskdir} generated by $s$. 
We define $\pdol_{\CS_+}(s)=0$ (i.e.\ $\Gamma_{-1,1}=0$), giving this sub-bimodule a holomorphic structure. We use (\ref{sideswitch}) to switch the sides of the previous spinor inner product:
\begin{align*}
\< a\,s,\overline{b\,s}\>=&\ \<\overline{\J(a\,s)},\J^{-1}(b\,s)\>=\delta^2\,\<\overline{\bar s\,a^*},\bar s\,b^*\>
=\delta^2\,\mu\,a\,w\,b^*\ .
\end{align*}
Following the previous method, we have $\<s,\bar s\>=\delta^2\,\mu\,w$, so $g^{11}=\delta^2\,\mu\,w$. Then the Chern connection has its other 
Christoffel symbol
\[
\Gamma_{+1,1}=-\pd(w)w^{-1}=\bar z.\extd z.w^{-1}\ .
\]
This is the natural Chern connection here but note that it is not the bimodule connection used for the construction of the Dirac operator.

\end{document}